\documentclass[12pt]{article} 
\usepackage{amssymb}
\usepackage{amsmath,bm}
\usepackage{graphics}
\usepackage{color}
\usepackage[color,matrix,arrow]{xy}
\usepackage{xcolor}
\usepackage{xypic}
\usepackage{tikz}
\usepackage[utf8]{inputenc}                    
\usepackage{microtype}                         
\usepackage{mathtools, amsthm, amssymb, eucal} 
\usepackage[normalem]{ulem}                    
\usepackage{mdframed}
\usepackage{verbatim}
\usepackage{booktabs} 
\usepackage{mathrsfs}
\usepackage{titling}
\usepackage[paper=letterpaper, margin=1in, headsep=20pt]{geometry}
\usepackage{hyperref}
\usepackage{enumerate}
\usepackage{multicol}
    \theoremstyle{plain}
    \newtheorem{thm}{Theorem}[section]
    \newtheorem{lem}[thm]{Lemma}
    
    \newtheorem{cor}[thm]{Corollary}

    \theoremstyle{definition}
    \newtheorem{defn}{Definition}[section]
    \newtheorem{conj}{Conjecture}[section]
    \newtheorem{exmp}{Example}[section]

    \theoremstyle{remark}
    \newtheorem{rem}{Remark}[section]

\input prepictex   \input pictex    \input postpictex

\def\mput #1 #2/{\put{$\bullet_{#1}$} [tl] <-1mm,1mm> at #2}

\def\iff{\Longleftrightarrow}

\begin{document}

\centerline{\textbf{\Large Five Lectures on Cluster Theory}}
\centerline{Ray Maresca}
\medskip

\begin{abstract}

\noindent
In this paper, we will present the author's interpretation and embellishment of five lectures on cluster theory given by Kiyoshi Igusa during the Spring semester of 2022 at Brandeis University. They are meant to be used as an introduction to cluster theory from a representation-theoretic point of view. 
\end{abstract}

\section{Introduction}

\subsection{Some History}

\indent

Since its introduction in the early 2000's by Fomin and Zelevinsky, cluster theory has been an active area of research. Some of the first results in cluster theory, such Fomin and Zelevinsky's classification of finite type cluster algebras in [\ref{ref: FZ finite type classification}], bore striking resemblances to theorems in representation theory such as Gabriel's theorem in [\ref{ref: Gabriel's theorem 1}] and [\ref{ref: Gabriel's theorem 2}]. In particular, every seed or cluster in a cluster algebra has some number, say $n$, cluster variables. This allows us to represent the variables of a cluster algebra in a graph in which there is an edge between any two variables that occur in a cluster. In this graph, which we will see in Section \ref{sec: clusters with modules}, we see that if we have $n-1$ cluster variables connected by edges, there are precisely 2 ways to complete the the graph. \\

On the other hand, many of these phenomena were also being seen in representation theory. For instance in [\ref{ref: S tilting modules}], Skowro\'{n}ski showed that every basic tilting module has $n$ indecomposable summands. Moreover, Happel and Unger showed in [\ref{ref: HU two completions}] that every basic partial tilting module having $n -1$ indecomposable direct summands can always be completed into a tilting module and that there are at most two ways that this can be done. Therefore, one may think that the `right' way to connect cluster theory and representation theory is through tilting theory, though this fails in several ways, one being that there are more cluster variables then there are indecomposable rigid modules and more clusters than tilting modules. \\

To attain a categorification of cluster theory in terms of representation theory, we thus need to extend the module category in some way. This was done by Buan, Marsh, Reineke, Reiten, and Todorov in [\ref{ref: BMRRT categorification of cluster algebras}] where they constructed a larger category called the cluster category in which the original module category can be embedded. In [\ref{ref: BMRRT categorification of cluster algebras}], it was assumed that the quiver $Q$ associated to the initial seed was acyclic, so the path algebra is hereditary, which is an assumption we will also make throughout this paper. This need not be the case and in [\ref{ref: A non hereditary cluster cat}], Amoit removes the condition of $\Bbbk Q$ being hereditary and constructs the cluster category for non-hereditary algebras of global dimension 2 and quivers with potential. \\

There however is still one thing to notice. In the cluster category of a hereditary algebra, we have the following isomorphism from Auslander-Reiten duality: Ext$^1(M,N) \cong $ Hom$(N,\tau M)$. One can show that this isomorphism actually also holds in module categories of non-hereditary cluster algebras. Therefore, there is a correspondence between the tilting objects in the cluster category and modules in the module category of a cluster-tilted algebra. The issue with this is that these modules may not be partial tilting objects due to having infinite projective dimension. By dropping the requirement on projective dimension and loosening rigidity to $\tau$-rigidity, Adachi, Iyama, and Reiten introduced $\tau$-tilting theory in [\ref{ref: AIR tau-tilting}].\\

\subsection{Framework of These Notes}

\indent

Although $\tau$-tilting theory is one of the most active areas of current research, in this paper, we will focus on classical tilting theory and cluster theory. For a survey on $\tau$-tilting theory, we suggest [\ref{ref: T survey}] by Treffinger, where many of the references and much of the background information in these notes were found. In this article, we will illuminate some connections between cluster theory and representation theory while working through the process of categorifying cluster theory when the initial seed corresponds to a quiver with no loops or two cycles. We will do this by working through the author's interpretation and embellishment of 5 lectures given by Kiyoshi Igusa during the spring semester of 2022 at Brandeis University which contain several motivating examples and provide some intuition behind results. One thing to note is that all proofs in the first six sections of this paper are meant to provide the main idea and intuition behind the proof and should be taken as nothing more than sketches. We will assume that the reader has some background in the foundations of representation theory and suggest [\ref{ref: blue book}], [\ref{ref: Bill's notes}], [\ref{ref: Schiffler Quiver Reps}], and [\ref{ref: S general reps}] as references for this material. \\

We will begin these notes with Section \ref{sec: cluster theory} in which we give some examples that provide intuition behind what a cluster algebra is and how clusters and cluster variables are connected to representation theory. We then provide definitions of cluster algebras and mutations in terms of a quiver $Q$. In particular, we will see a connection between cluster variables (characters) and the Auslander-Reiten (AR) quiver of the corresponding initial quiver. Afterward in Section \ref{sec: cluster character} we will explicitly provide the correspondence by showing how to associate a cluster character to a $\Bbbk Q$-module. We will moreover see how the coefficients of this cluster character are related to the module itself. \\

In Section \ref{sec: clusters with modules} we will introduce two questions that we will spend the rest of the notes trying to answer; namely, which sets of modules are sent to clusters and which algebraic objects correspond to the initial cluster variables? These two questions motivated the categorification of cluster theory in terms of representation theory. It is here we will introduce the notion of tilting modules and how they fall short of describing cluster theory in its entirety. We will need to extend the idea of a tilting module by introducing support tilting modules, shifted projectives, and silting pairs. We will do this without introducing the bounded derived category or explicitly creating the cluster category. At this point, we will provide the bijection between clusters and silting pairs. We will finish this section by constructing the wall and chamber structure (stability picture) and show how this structure connects to cluster theory. \\

After this, in Section \ref{sec: rigidity}, we will describe in fact why the stability picture introduced in the previous section is accurate by showing that rigidity is a Zariski open condition. To do this, we will introduce the category of 2-term silting complexes. We finish this section by introducing the notion of stable barcodes which won't actually be used for the remainder of the notes. In the final section, Section \ref{sec: maximal green sequences}, we introduce the notion of maximal green sequences using exchange matrices and ice quivers. The definition of these sequences rely on notions like sign coherence of $g$-vectors and $c$-vectors, which we will also explain in this section. Throughout the notes, we will attempt to provide as much referencing as possible to both history and proofs of the results. \\

Before we begin, we would like to remark that not all connections between cluster and representation theory are made in these notes. For instance, there is a beautiful connection between functorially finite torsion classes in mod-$\Bbbk Q$ and clusters, namely that they are in bijection. One way to see this is through the fact that the dual graph of the stability picture is precisely the Hasse quiver of functorially finite torsion classes in mod-$\Bbbk Q$. For more on the lattice of torsion classes, we suggest Thomas's exposition [\ref{ref: Th survey on lattice of torsion classes}]. For more details on the bijection, Treffinger's survey [\ref{ref: T survey}] is a great place to start.


\section{Cluster Theory}\label{sec: cluster theory}

\subsection{Examples}\label{sec: Examples}

\indent

Before presenting the formal definition of a cluster algebra, we provide some intuition, motivation, and examples. Intuitively, `clusters' are sets of $n$ objects which can be mutated. Each of these $n$ objects are a cluster variable or cluster character. Throughout these lectures, we will provide two methods of thinking about clusters, the former is the original definition and the later is a categorification of it. 

\begin{enumerate}[(a)]
\item Clusters are transcendence bases for $\mathbb{Q}(x_1, x_2, \dots , x_n)$ given by a quiver $Q$. \\
\item Clusters are objects in a category of the form $T = T_1 \oplus T_2 \oplus \dots \oplus T_n$.
\end{enumerate}

One relationship between the above two methods is through something called the `cluster character' denoted by $\chi(T_i) \in \mathbb{Q}(x_1, x_2, \dots , x_n)$. We begin our studies of cluster algebras using method (a) through an example. \\

Until otherwise stated, let $Q$ be the quiver $2 \rightarrow 1 \leftarrow 3$. Note that this quiver consists of only \textbf{descending} arrows, that is, there is an arrow from $i$ to $j$ if and only if $j < i$. Below are three depictions of the Auslander-Reiten (AR) quiver for this quiver. Note that the maps between the projectives are ascending with respect to their vertices; that is, there is an arrow from $P_i$ to $P_j$ in the AR quiver if and only if $i < j$. This is one reason to always take the arrows to be descending when the quiver does not have oriented cycles. Moreover, note that the quiver formed by the three projectives in the AR quiver is the opposite quiver of $Q$. On the top left we have the standard projective/injective at vertex $i$ notation. On the top right we have another standard notation indicating the tops and socles of the modules on the left. Finally, below these two is a depiction of the AR quiver using dimension vectors. In each quiver the dotted lines indicate the AR translate $\tau$:

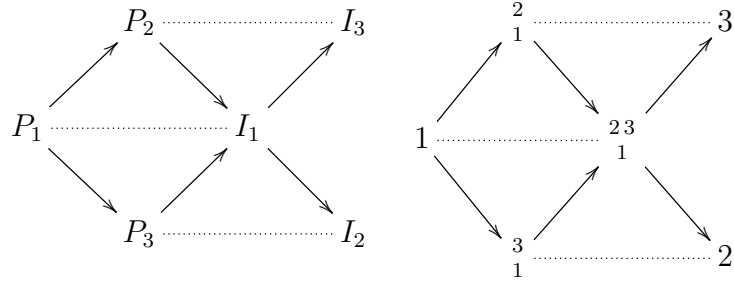
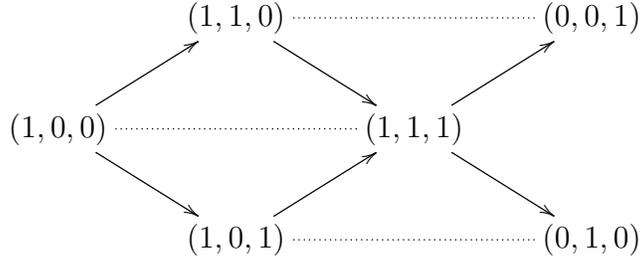
\begin{center}
\begin{tabular}{c c}
\xymatrix{                             & P_2 \ar[dr] \ar@{.}[rr]  &                             & I_3 \\
				P_1 \ar[ur] \ar[dr] \ar@{.}[rr]  &                  & I_1 \ar[ur] \ar[dr] &       \\
											& P_3 \ar[ur] \ar@{.}[rr]  &                             & I_2} 
 & 
\xymatrix{                             & {2 \atop 1} \ar[dr] \ar@{.}[rr]  &                             & 3 \\
				1 \ar[ur] \ar[dr] \ar@{.}[rr]  &                  & {2 \, 3 \atop 1} \ar[ur] \ar[dr] &       \\
											& {3 \atop 1} \ar[ur] \ar@{.}[rr]  &                             & 2} 
\end{tabular}
\end{center}

\begin{center}
\begin{tabular}{c}
\xymatrix{                             & (1,1,0) \ar[dr] \ar@{.}[rr]  &                             & (0,0,1) \\
				(1,0,0) \ar[ur] \ar[dr] \ar@{.}[rr]  &                  & (1,1,1) \ar[ur] \ar[dr] &       \\
											& (1,0,1) \ar[ur] \ar@{.}[rr]  &                             & (0,1,0)} 
\end{tabular}
\end{center}

The form that uses the dimension vectors is computationally useful and sheds some light onto the relationship between AR theory and cluster theory. Recall that given a short exact sequence of modules $0\rightarrow A \rightarrow B \rightarrow C \rightarrow 0$, the dimension vectors satisfy the equality dim$A \, + $ dim$C = $ dim$B$. Since each mesh in the above AR quiver forms an almost split sequence, which is short exact, we can construct the AR quiver using this relationship and in some sense extend it as follows.

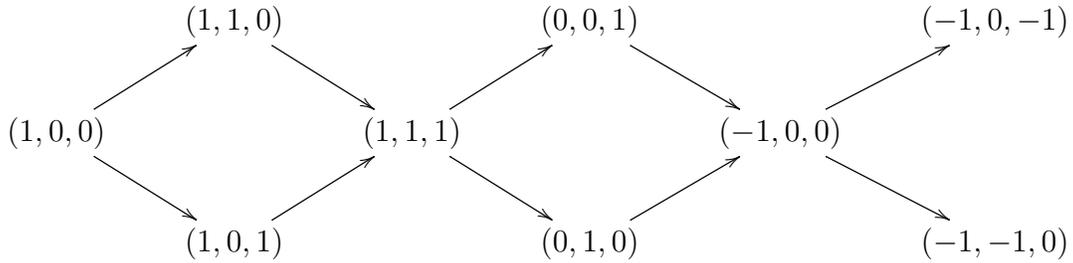
\begin{center}
\begin{tabular}{c}
\xymatrix{                             & (1,1,0) \ar[dr]  &                             & (0,0,1) \ar[dr] &  &(-1,0,-1)\\
				(1,0,0) \ar[ur] \ar[dr]  &                  & (1,1,1) \ar[ur] \ar[dr] &     & (-1,0,0) \ar[ur] \ar[dr]  \\
											& (1,0,1) \ar[ur]  &                             & (0,1,0) \ar[ur] & &(-1, -1, 0)} 
\end{tabular}
\end{center}

These newly added vectors are not dimension vectors in the usual sense since they contain negative entries. From a representation theoretic point of view, we will see in Section \ref{sec: clusters with modules} that these negative `dimension vectors' correspond to `shifted projectives' in the categorification method of studying cluster algebras. Now to the dimension vector $(i_1,i_2,\dots , i_n)$, we associate the symbol $x_1^{i_1}x_2^{i_2}\dots x_n^{i_n}$. Then from the aforementioned additive relationship given by the dimension vectors, we have that given a short exact sequence of modules $0\rightarrow A \rightarrow B \rightarrow C \rightarrow 0$, the symbols satisfy the equality $x^{\text{dim}A}x^{\text{dim}C} = x^{\text{dim}B}$. In this notation, by $x^{\text{dim}A}$ we mean $\prod_{i=1}^n x_i^{\text{dim}A_i}$ where $\text{dim}A_i$ is the $i$th entry in the dimension vector of $A$. Then the above AR quiver can be rewritten in terms of the symbols as follows.

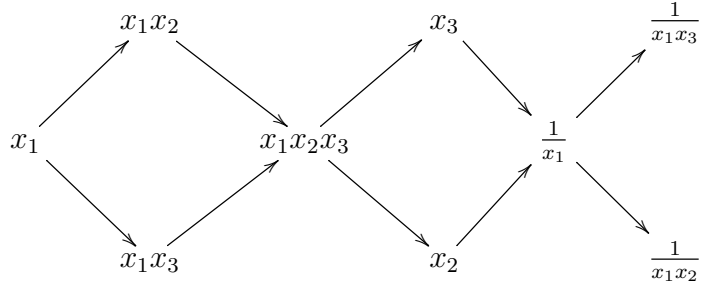
\begin{center}
\begin{tabular}{c}
\xymatrix{                             & x_1x_2 \ar[dr]  &                             & x_3 \ar[dr] &  &{1\over x_1x_3}\\
				x_1 \ar[ur] \ar[dr]  &                  & x_1x_2x_3 \ar[ur] \ar[dr] &     & {1\over x_1} \ar[ur] \ar[dr]  \\
											& x_1x_3 \ar[ur]  &                             & x_2 \ar[ur] & &{1\over x_1x_2}} 
\end{tabular}
\end{center}

\subsection{Cluster Variables}\label{sec: Cluster Variables}

\indent 

These symbols are the cluster variables or cluster characters for the cluster algebra whose initial quiver is $Q$. Though we have not yet defined the cluster variables, they satisfy a nice property. As was shown by Caldero and Chapoton in [\ref{ref: CC equation}], given an almost split sequence of modules $0\rightarrow A \rightarrow B \rightarrow C \rightarrow 0$, the corresponding cluster characters satisfy the relationship $$\chi(A)\chi(C) = \chi(B) + 1$$ where $\chi(B) = \chi(\oplus_i B_i) = \sum_i \chi(B_i)$. Note that if the sequence is neither split nor almost split, we would need to add more than one on the right hand side and if the sequence is split, we would not need to add anything, providing some intuition on why the sequences are almost split but not split. We will see in Section \ref{sec: cluster character} where this plus 1 is coming from. Using this formula for cluster characters, we can reconstruct the AR quiver for our type $\mathbb{A}$ quiver $Q$, along with those `shifted projectives', for the corresponding path algebra by beginning with the opposite quiver as follows.

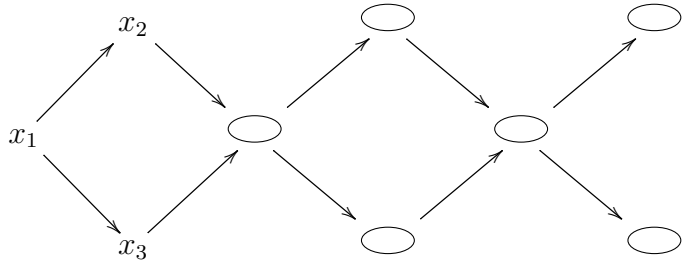
\begin{center}
\begin{tabular}{c}
\xymatrix{                             & x_2 \ar[dr]  &                             & \tikz \draw (0,0) ellipse (10pt and 5pt); \ar[dr] &  &\tikz \draw (0,0) ellipse (10pt and 5pt);\\
				x_1 \ar[ur] \ar[dr]  &                  & \tikz \draw (0,0) ellipse (10pt and 5pt); \ar[ur] \ar[dr] &     & \tikz \draw (0,0) ellipse (10pt and 5pt); \ar[ur] \ar[dr]  \\
											& x_3 \ar[ur]  &                             & \tikz \draw (0,0) ellipse (10pt and 5pt); \ar[ur] & &\tikz \draw (0,0) ellipse (10pt and 5pt);} 
\end{tabular}
\end{center}

To fill the left most oval, we need a cluster variable $y$ such that $x_1y = x_2x_3 + 1$, so $y = {x_2x_3 + 1 \over x_1}$. We get:

 \begin{center}
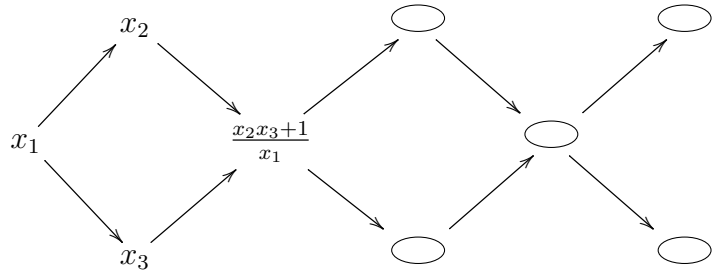
\begin{tabular}{c}
\xymatrix{                             & x_2 \ar[dr]  &                             & \tikz \draw (0,0) ellipse (10pt and 5pt); \ar[dr] &  &\tikz \draw (0,0) ellipse (10pt and 5pt);\\
				x_1 \ar[ur] \ar[dr]  &                  & {x_2x_3 + 1 \over x_1} \ar[ur] \ar[dr] &     & \tikz \draw (0,0) ellipse (10pt and 5pt); \ar[ur] \ar[dr]  \\
											& x_3 \ar[ur]  &                             & \tikz \draw (0,0) ellipse (10pt and 5pt); \ar[ur] & &\tikz \draw (0,0) ellipse (10pt and 5pt);} 
\end{tabular}
\end{center}

Continuing in this way will provide us the entire AR quiver and more for $Q$ in terms of the cluster variables, though it will get quite messy quickly. To reduce unnecessary computations, we examine a simpler example. Let $Q$ now be the quiver $1 \leftarrow 2$. Then the AR quiver is 

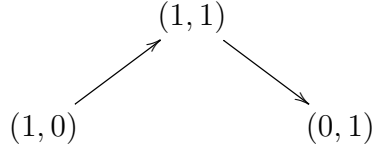
\begin{center}
\begin{tabular}{c}
\xymatrix{  & (1,1) \ar[dr] \\
				(1,0) \ar[ur] & & (0,1)} 
\end{tabular}
\end{center}

By using the relationship between cluster characters and beginning with the opposite quiver $x_1 \rightarrow x_2$, we get the following.

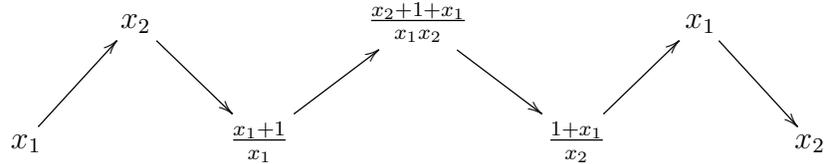
\begin{center}
\begin{tabular}{c}
\xymatrix{                  & x_2 \ar[dr] &                                          & {x_2 +1+x_1\over x_1x_2} \ar[dr] &                                        & x_1 \ar[dr] &\\
				x_1 \ar[ur] &                   & {x_1 + 1\over x_1} \ar[ur] &                                                      & {1+x_1\over x_2} \ar[ur] &                  & x_2 } 
\end{tabular}
\end{center}

Notice that if we restrict our attention to only the cluster variables with nontrivial denominators, we are looking at the AR quiver of $Q$. These cluster characters correspond to $\Bbbk Q$-modules and the intimate relationship between the dimension vectors and the cluster variables will be revealed in Section \ref{sec: cluster character}. The cluster variables with trivial denominator correspond to the shifted projectives which will be explained in Section \ref{sec: clusters with modules}.

\subsection{Cluster Algebras and Mutation}

\indent

Now that we have a sense of what we want cluster variables to be, we are ready to provide a formal definition. The definitions of a cluster algebra and mutation was first written down by Fomin and Zelevinski in [\ref{ref: FZ cluster algbras}]. The original definition used the notion of an exchange matrix instead of a quiver; however as we will soon see, these two notions are one in the same.

\begin{defn}
 A \textbf{cluster algebra} is a subalgebra $A\subset \mathbb{Q}(x_1,x_2,\dots, x_n)$ generated by the cluster variables given by mutating a seed $(Q,x_{*})$ where
\begin{itemize}
\item $Q$ is a quiver with no loops or two cycles. 
\item $x_{*} = (x_1,x_2,\dots,x_n)$ is a transcendence basis for $\mathbb{Q}(x_1,x_2,\dots, x_n) = \{{f(x)\over g(x)} : f,g \in \mathbb{Q}(x_1,x_2,\dots, x_n) \}$. 
\end{itemize} 
\end{defn}

In order for this definition to be complete, we must define mutation which consists of two parts, mutation of the quiver and the seed. We begin by defining mutation of a quiver $Q$ with a running example. For the following definition, let $Q'$ be the quiver 

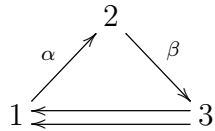
\begin{center}
\begin{tabular}{c}
\xymatrix{                             & 2 \ar[dr]^{\beta}   &                             \\
				1 \ar[ur]^{\alpha}     &                  & 3 \ar@<.5ex>[ll] \ar@<-.5ex>[ll]       }
\end{tabular}
\end{center}

\begin{defn} \label{def: quiver mutation}
We define the \textbf{mutation of $Q$ at vertex $k$}, denoted by $\mu_kQ$, as follows.

\begin{enumerate}
\item We first compose any length two paths through $k$ by introducing a new arrow. In the running example, let $k=2$.

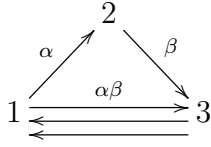
\begin{center}
\begin{tabular}{c}
\xymatrix{                             & 2 \ar[dr]^{\beta}   &                             \\
				                                            1 \ar[ur]^{\alpha} \ar@<.5ex>[rr]^{\alpha\beta}    &                  & 3 \ar@<.5ex>[ll] \ar@<1.5ex>[ll]        }
\end{tabular}
\end{center}
\item Reverse all arrows at $k$:
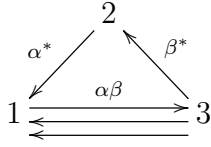
\begin{center}
\begin{tabular}{c}
\xymatrix{                             & 2 \ar[dl]_{\alpha^*}   &                             \\
				                                            1  \ar@<.5ex>[rr]^{\alpha\beta}    &                  & 3 \ar@<.5ex>[ll] \ar@<1.5ex>[ll] \ar[ul]_{\beta^*}       }
\end{tabular}
\end{center}

\item Eliminate all two cycles:

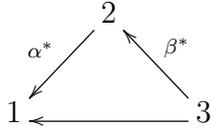
\begin{center}
\begin{tabular}{c}
\xymatrix{                             & 2 \ar[dl]_{\alpha^*}   &                             \\
				                                            1    &                  & 3 \ar@<.5ex>[ll] \ar[ul]_{\beta^*}       }
\end{tabular}
\end{center}

\end{enumerate}

After all three steps, the resulting quiver is $\mu_kQ$, or in this example $\mu_2Q'$. 
\end{defn}
Notice that the procedure of mutation does not preserve representation type of $Q$. To provide another example of mutation:

\begin{exmp}\label{exmp: mutation of type A quiver}
Let $Q$ be the quiver $2\rightarrow 1\leftarrow 3$. Then $\mu_1Q$ is the quiver $2 \leftarrow 1 \rightarrow 3$.
\end{exmp}

Now is a good time to introduce the notion of exchange matrices, which will become important in the study of maximal green sequences which will be defined in Section \ref{sec: maximal green sequences}. 

\begin{defn} \label{defn: exchange matrix}
Let $Q$ be a quiver. The \textbf{exchange matrix} of $Q$ is the matrix $B = [b_{ij}]$ where $b_{ij} =$ the number of arrows $i\rightarrow j$ - the number of arrows $j \rightarrow i$.
\end{defn} 

\begin{exmp}\label{exmp: exchange matrix}
For the running example in Definition \ref{def: quiver mutation}, the exchange matrix is $$\left[ \vcenter{\xymatrixrowsep{10pt}\xymatrixcolsep{10pt}\xymatrix{0&1&-1\\-1&0&2\\1&-2&0}} \right]$$
\end{exmp}

Notice that the exchange matrix in Example \ref{exmp: exchange matrix} is skew symmetric, that is, all entries are integers and $b_{ij} = -b_{ji}$. This holds more generally; in fact, any skew symmetric matrix gives a quiver with no loops or two cycles and any quiver with no loops or two-cycles has a skew symmetric exchange matrix, so really in our definition of cluster algebra, we are only considering `skew symmetric' cluster algebras. We are now ready to define mutation of the seed $x_*$.

\begin{defn}\label{defn: mutation of seed}
We define the \textbf{mutation of $x_*$ at vertex $k$}, denoted by $\mu_k(x_*)$, as $\mu_k(x_*) = (x_1, x_2, \dots , x'_k, x_{k+1}, \dots, x_n)$ where $$x'_k = {\displaystyle \prod_{i\rightarrow k} x_i^{b_{ik}} + \displaystyle \prod_{k \rightarrow j} x_j^{b_{kj}} \over x_k}.$$
The notation $\displaystyle \prod_{i\rightarrow k}$ means we take the product over all vertices $i$ such that there is an arrow from $i$ to $k$. We adopt the standard convention that the empty product equals 1.
\end{defn}

\begin{exmp}\label{exmp: mutation of mutated seed}
Let  $Q$ be the quiver $2\rightarrow 1\leftarrow 3$ and consider the seed $(x_1,x_2,x_3)$. Then mutation at vertex 1 gives $\mu_1(Q,(x_1,x_2,x_3)) = (\mu_1Q, ({x_2x_3 + 1 \over x_1}, x_2, x_3))$ with $\mu_1Q$ the quiver from Example \ref{exmp: mutation of type A quiver}. Notice that the plus one in the mutated seed comes from the fact that vertex 1 in $Q$ is a sink, so there are no arrows out of it, and the empty product is always taken to be 1. Continuing by mutating at vertex 2 then 3 gives the following: 

\begin{center}
\begin{tabular}{c}
\xymatrix{           ( 2 \leftarrow 1 \rightarrow 3,  ({x_2x_3 + 1 \over x_1}, x_2, x_3))  \ar[d]_{\mu_2} \\
	                      ( 2 \rightarrow 1 \rightarrow 3,  ({x_2x_3 + 1 \over x_1}, {x_2x_3 + 1 + x_1 \over x_1x_2}, x_3))  \ar[d]_{\mu_3} \\
                          ( 2 \rightarrow 1 \leftarrow 3,  ({x_2x_3 + 1 \over x_1}, {x_2x_3 + 1 + x_1 \over x_1x_2}, {x_2x_3 + 1 + x_1 \over x_1x_3}))   }
\end{tabular}
\end{center}

We can visualize the mutations of these cluster variables as follows. Recall the construction of the AR quiver for $Q: 2 \rightarrow 1 \leftarrow3$ from Section \ref{sec: Cluster Variables}. Continuing this procedure would give the picture in Figure \ref{fig: A visualization of mutation in the AR quiver}. Within each loop we have the opposite quiver of the mutated quiver along with the mutated seed after performing the indicated mutation. Moreover, each loop encloses a \textbf{cluster}; that is, a collection of $n$ cluster variables attained by mutations.

\begin{center}
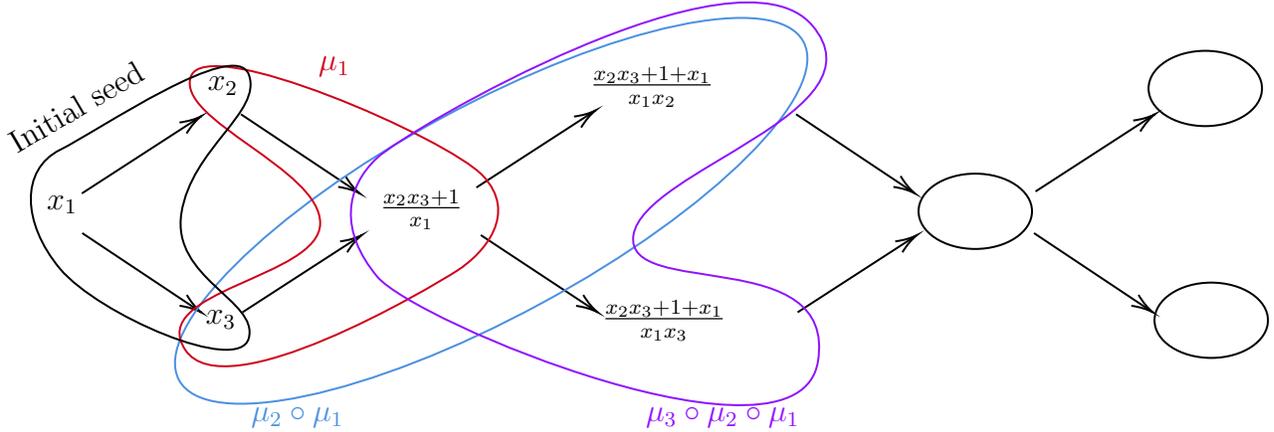
\begin{figure}[h!]

\tikzset{every picture/.style={line width=0.75pt}} 

\begin{tikzpicture}[x=0.75pt,y=0.75pt,yscale=-1,xscale=1]

\draw    (42.22,158.01) -- (100.56,196.9) ;
\draw [shift={(102.22,198.01)}, rotate = 213.69] [color={rgb, 255:red, 0; green, 0; blue, 0 }  ][line width=0.75]    (10.93,-3.29) .. controls (6.95,-1.4) and (3.31,-0.3) .. (0,0) .. controls (3.31,0.3) and (6.95,1.4) .. (10.93,3.29)   ;
\draw    (42,138) -- (100.54,100.09) ;
\draw [shift={(102.22,99.01)}, rotate = 147.08] [color={rgb, 255:red, 0; green, 0; blue, 0 }  ][line width=0.75]    (10.93,-3.29) .. controls (6.95,-1.4) and (3.31,-0.3) .. (0,0) .. controls (3.31,0.3) and (6.95,1.4) .. (10.93,3.29)   ;
\draw    (122.22,98.01) -- (180.56,136.9) ;
\draw [shift={(182.22,138.01)}, rotate = 213.69] [color={rgb, 255:red, 0; green, 0; blue, 0 }  ][line width=0.75]    (10.93,-3.29) .. controls (6.95,-1.4) and (3.31,-0.3) .. (0,0) .. controls (3.31,0.3) and (6.95,1.4) .. (10.93,3.29)   ;
\draw    (123,198) -- (181.54,160.09) ;
\draw [shift={(183.22,159.01)}, rotate = 147.08] [color={rgb, 255:red, 0; green, 0; blue, 0 }  ][line width=0.75]    (10.93,-3.29) .. controls (6.95,-1.4) and (3.31,-0.3) .. (0,0) .. controls (3.31,0.3) and (6.95,1.4) .. (10.93,3.29)   ;
\draw    (241,135) -- (299.54,97.09) ;
\draw [shift={(301.22,96.01)}, rotate = 147.08] [color={rgb, 255:red, 0; green, 0; blue, 0 }  ][line width=0.75]    (10.93,-3.29) .. controls (6.95,-1.4) and (3.31,-0.3) .. (0,0) .. controls (3.31,0.3) and (6.95,1.4) .. (10.93,3.29)   ;
\draw    (243.22,159.01) -- (300.55,196.9) ;
\draw [shift={(302.22,198.01)}, rotate = 213.47] [color={rgb, 255:red, 0; green, 0; blue, 0 }  ][line width=0.75]    (10.93,-3.29) .. controls (6.95,-1.4) and (3.31,-0.3) .. (0,0) .. controls (3.31,0.3) and (6.95,1.4) .. (10.93,3.29)   ;
\draw    (403,198) -- (461.54,160.09) ;
\draw [shift={(463.22,159.01)}, rotate = 147.08] [color={rgb, 255:red, 0; green, 0; blue, 0 }  ][line width=0.75]    (10.93,-3.29) .. controls (6.95,-1.4) and (3.31,-0.3) .. (0,0) .. controls (3.31,0.3) and (6.95,1.4) .. (10.93,3.29)   ;
\draw    (402.22,98.01) -- (460.56,136.9) ;
\draw [shift={(462.22,138.01)}, rotate = 213.69] [color={rgb, 255:red, 0; green, 0; blue, 0 }  ][line width=0.75]    (10.93,-3.29) .. controls (6.95,-1.4) and (3.31,-0.3) .. (0,0) .. controls (3.31,0.3) and (6.95,1.4) .. (10.93,3.29)   ;
\draw   (464,147) .. controls (464,136.51) and (476.81,128) .. (492.61,128) .. controls (508.41,128) and (521.22,136.51) .. (521.22,147) .. controls (521.22,157.5) and (508.41,166.01) .. (492.61,166.01) .. controls (476.81,166.01) and (464,157.5) .. (464,147) -- cycle ;
\draw    (522.22,158.01) -- (580.56,196.9) ;
\draw [shift={(582.22,198.01)}, rotate = 213.69] [color={rgb, 255:red, 0; green, 0; blue, 0 }  ][line width=0.75]    (10.93,-3.29) .. controls (6.95,-1.4) and (3.31,-0.3) .. (0,0) .. controls (3.31,0.3) and (6.95,1.4) .. (10.93,3.29)   ;
\draw    (523,137) -- (581.54,99.09) ;
\draw [shift={(583.22,98.01)}, rotate = 147.08] [color={rgb, 255:red, 0; green, 0; blue, 0 }  ][line width=0.75]    (10.93,-3.29) .. controls (6.95,-1.4) and (3.31,-0.3) .. (0,0) .. controls (3.31,0.3) and (6.95,1.4) .. (10.93,3.29)   ;
\draw   (580,85) .. controls (580,74.51) and (592.81,66) .. (608.61,66) .. controls (624.41,66) and (637.22,74.51) .. (637.22,85) .. controls (637.22,95.5) and (624.41,104.01) .. (608.61,104.01) .. controls (592.81,104.01) and (580,95.5) .. (580,85) -- cycle ;
\draw   (583,202) .. controls (583,191.51) and (595.81,183) .. (611.61,183) .. controls (627.41,183) and (640.22,191.51) .. (640.22,202) .. controls (640.22,212.5) and (627.41,221.01) .. (611.61,221.01) .. controls (595.81,221.01) and (583,212.5) .. (583,202) -- cycle ;
\draw  [color={rgb, 255:red, 74; green, 144; blue, 226 }  ,draw opacity=1 ] (91.03,232.21) .. controls (77.04,206.43) and (136.18,147.28) .. (223.12,100.09) .. controls (310.07,52.9) and (391.9,35.55) .. (405.89,61.33) .. controls (419.88,87.11) and (360.74,146.26) .. (273.79,193.45) .. controls (186.85,240.63) and (105.02,257.99) .. (91.03,232.21) -- cycle ;
\draw  [color={rgb, 255:red, 208; green, 2; blue, 27 }  ,draw opacity=1 ] (97,79) .. controls (104.78,56.99) and (226.22,110.01) .. (244.22,129.01) .. controls (262.22,148.01) and (245.22,168.01) .. (231.22,177.01) .. controls (217.22,186.01) and (114.22,247.01) .. (94.22,217.01) .. controls (74.22,187.01) and (162.26,180.59) .. (162.22,153.01) .. controls (162.18,125.42) and (89.22,101.01) .. (97,79) -- cycle ;
\draw   (31,116) .. controls (51,106) and (119.22,59.01) .. (126.22,78.01) .. controls (133.22,97.01) and (97.34,115.07) .. (92.22,148.01) .. controls (87.1,180.94) and (131.22,192.01) .. (126.22,211.01) .. controls (121.22,230.01) and (50.22,199.01) .. (31,176) .. controls (11.78,152.99) and (11,126) .. (31,116) -- cycle ;
\draw  [color={rgb, 255:red, 144; green, 19; blue, 254 }  ,draw opacity=1 ] (191,120) .. controls (207.22,105.01) and (372.22,9.01) .. (410.22,53.01) .. controls (448.22,97.01) and (325.34,126.07) .. (320.22,159.01) .. controls (315.1,191.94) and (423.22,163.01) .. (413.22,223.01) .. controls (403.22,283.01) and (210.22,203.01) .. (191,180) .. controls (171.78,156.99) and (174.78,134.99) .. (191,120) -- cycle ;

\draw (23,137.4) node [anchor=north west][inner sep=0.75pt]    {$x_{1}$};
\draw (104,77.4) node [anchor=north west][inner sep=0.75pt]    {$x_{2}$};
\draw (103,195.4) node [anchor=north west][inner sep=0.75pt]    {$x_{3}$};
\draw (191,134.4) node [anchor=north west][inner sep=0.75pt]    {$\frac{x_{2} x_{3} +1}{x_{1}}$};
\draw (297,72.4) node [anchor=north west][inner sep=0.75pt]    {$\frac{x_{2} x_{3} +1+x_{1}}{x_{1} x_{2}}$};
\draw (303,190.4) node [anchor=north west][inner sep=0.75pt]    {$\frac{x_{2} x_{3} +1+x_{1}}{x_{1} x_{3}}$};
\draw (160,67.4) node [anchor=north west][inner sep=0.75pt]  [color={rgb, 255:red, 208; green, 2; blue, 27 }  ,opacity=1 ]  {$\mu _{1}$};
\draw (1.77,107.83) node [anchor=north west][inner sep=0.75pt]  [rotate=-329.69] [align=left] {Initial seed};
\draw (126,244.4) node [anchor=north west][inner sep=0.75pt]  [color={rgb, 255:red, 74; green, 144; blue, 226 }  ,opacity=1 ]  {$\mu _{2} \circ \mu _{1}$};
\draw (325,244.34) node [anchor=north west][inner sep=0.75pt]  [color={rgb, 255:red, 144; green, 19; blue, 254 }  ,opacity=1 ]  {$\mu _{3} \circ \mu _{2} \circ \mu _{1}$};

\end{tikzpicture}
\caption{A visualization of mutation in the AR quiver}
\label{fig: A visualization of mutation in the AR quiver}
\end{figure}

\end{center} 

\end{exmp}
\vspace{-1 cm}
In all of the examples we have done so far, the cluster variables were all Laurent polynomials; that is, they are of the form ${f(x)\over x^{\alpha}}$ where $f(x)$ is a polynomial in $\mathbb{N}(x_1, x_2, \dots , x_n)$ and $\alpha \in \mathbb{Z}^n$. It turns out that whether or not this phenomenon holds in general has been an open question called the positivity conjecture first stated by Fomin and Zelevinsky in 2002 in [\ref{ref: FZ cluster algbras}]. Around 2015 in [\ref{ref: LS positivity}], Lee and Schiffler proved the positivity conjecture for all cluster algebras for which the exchange matrix is skew symmetric, called \textbf{skew symmetric} cluster algebras, which is the convention we take in these notes for all our cluster algebras. The general case was proven in 2017 by Gross, Hacking, Keel, and Kontsevich in [\ref{ref: GHKK sign coherence c-vectors}]. We have the following theorem.

\begin{thm}{\color{white}.}\label{Thm: cluster variables are Laurent Polynomials}
\begin{itemize}
\item Cluster variables are Laurant polynomials.
\item The numerator of the Laurant polynomial always has nonnegative integral coefficients.
\end{itemize}
\end{thm}

The fact that the coefficients of the numerator are nonnegative integers may lead us to believe that they count something. This is indeed the case as we will see in Section \ref{sec: cluster character}.

\section{Cluster Character} \label{sec: cluster character}

\indent 

In this section, we begin by showing how to attain the cluster character associated to a $\Bbbk Q$-module where $Q$ is a connected quiver with $n$ vertices and no oriented cycles; that is, $\Bbbk Q$ is a hereditary algebra which is an assumption we will take throughout the remainder of the notes unless otherwise specified. We then explore the connections between the cluster character and notions like quiver grassmannians and the Fomin-Zelevinsky formula for mutation of cluster variables. To do this, we need the notion of $g$-vectors, which arose from a purely cluster-algebraic perspective by Fomin and Zelevinsky in [\ref{ref: FZ g-vectors}]. Afterward, we realized that $g$-vectors can be studied in a representation-theoretic way through projective presentations. The first such realization was done by Dehy and Keller in [\ref{ref: DK rep theory g-vectors}]. Let $\Bbbk = \overline{\Bbbk}$ be an algebraically closed field. 

\begin{defn}{\color{white}.}
\begin{itemize}
\item The \textbf{$g$-vector} of a projective $\Bbbk Q$-module $P = \oplus a_iP_i$ is the vector $g(P) := \vec{a} = (a_1, a_2, \dots, a_n)$.
\item Let $M$ be a $\Bbbk Q$ module. Then $M$ admits a minimal projective presentation of the form $0\rightarrow P'_M\rightarrow P_M \rightarrow M \rightarrow 0$ where both $P'_M$ and $P_M$ are projective. The \textbf{$g$-vector} of $M$ is $g(M):= g(P_M) - g(P'_M)$.
\end{itemize}
\end{defn}

Now we provide some intuition behind how to think of the cluster character with an example.

\begin{exmp}
Let $Q = 2\rightarrow 1\leftarrow 3$. Referring back to the AR quiver given in Section \ref{sec: Examples} and Figure \ref{fig: A visualization of mutation in the AR quiver}, we see that the cluster character of $S_1 = P_1$ should be $\chi(S_1) = {x_2x_3 + 1\over x_1} = {x_2x_3\over x_1} + {1\over x_1}$. The module $S_1$ has two submodules, namely $0$ and $S_1$, and these should correspond to the left and right terms in the above sum respectively. The module $P_3 = {3 \atop 1}$ has three submodules: $0, S_1, P_3$. Again comparing the AR quiver and Figure \ref{fig: A visualization of mutation in the AR quiver}, we see that $\chi(P_3) = {x_{2} x_{3} +1+x_{1}\over x_{1} x_{3}} = {x_2\over x_1} + {1 \over x_1x_3} + {1\over x_3}$ and from left to right, these three terms should correspond to $0, S_1,$ and $P_3$ respectively.
\end{exmp}

To see how this correspondence works in general, we make the following definition which can be found in [\ref{ref: CC equation}].

\begin{defn}
The \textbf{cluster character} of the $\Bbbk Q$-module $M$ is $\chi(M) := \displaystyle \sum_{V \subset M} \chi(M,V)$, with $$\chi(M,V) := x^{-g(V)}x^{-g(D(M/V))}$$ where $D:\Bbbk Q$-mod$\rightarrow$ mod-$\Bbbk Q^{op}$ denotes the duality functor and $D(M/V)$ denotes the dual of the quotient module $M/V$.
\end{defn}

Let's now verify the results from the previous example.

\begin{exmp}\label{exmp: Computing Cluster Character}
We begin by computing the necessary $g$-vectors. Since $S_1 = P_1$ is projective, by definition we have $g(S_1) = (1,0,0)$. Similarly, $g(P_3) = (0,0,1)$. To compute $\chi(P_3)$ we also need $g(D(P_3/S_1)) = g(D(S_3))$ and $g(D(P_3/0)) = g(D(P_3))$. The former is easier than the later since $D(S_3)$ is the $\Bbbk Q^{op}$ representation $0 \leftarrow 0 \rightarrow \Bbbk$, hence it is a projective $\Bbbk Q^{op}$-module. Therefore $g(D(S_3)) = (0,0,1)$. To compute $g(D(P_3))$, note that $D(P_3))$ is the $\Bbbk Q^{op}$ representation $0 \leftarrow \Bbbk \rightarrow \Bbbk = {1\atop 3}$. We have a $\Bbbk Q^{op}$ minimal projective presentation given by $0\rightarrow 2 \rightarrow {1\atop 2 \, 3} \rightarrow {1 \atop 3} \rightarrow 0$. Thus $g(D(P_3)) = (1,0,0) - (0,1,0) = (1,-1,0)$. Finally, we need $g(D(S_1))$ to compute $\chi(S_1)$. Since $D(S_1)$ is the $\Bbbk Q^{op}$ representation $0 \leftarrow \Bbbk \rightarrow 0 = I_1$, we have a $\Bbbk Q^{op}$ minimal projective presentation given by $0\rightarrow 2 \oplus 3 \rightarrow {1\atop 2 \, 3} \rightarrow 1 \rightarrow 0$. Therefore $g(D(S_1)) = (1,0,0) - (0,1,0) - (0,0,1) = (1,-1,-1)$.  We are now ready to compute the cluster characters:
\vspace{-1cm}
\begin{multicols}{2}

\begin{align*}
\chi(P_3) &= \chi(P_3,0) + \chi(P_3,S_1) + \chi(P_3,P_3) \\
			  &=x^{0}x^{-g(D(P_3))} +  x^{-g(S_1)}x^{-g(D(S_3))} + x^{-g(P_3)}x^{-g(D(0))} \\
			  &= x^{0}x^{(-1,1,0)} +  x^{(-1,0,0)}x^{(0,0,-1)} + x^{(-1,0,0)}x^{0} \\
			  &={x_2\over x_1} + {1\over x_1x_3} + {1\over x_1}
\end{align*}

\columnbreak

\vspace{-2cm}
\begin{align*}
\chi(S_1) &= \chi(S_1,0) + \chi(S_1,S_1) \\
			  &=x^{0}x^{-g(D(S_1))} +  x^{-g(S_1)}x^{0} \\
			  &= x^{0}x^{(-1,1,1)} +  x^{(-1,0,0)}x^{0} \\
			  &={x_2x_3\over x_1} + {1\over x_1} 
\end{align*}

\end{multicols}

\end{exmp}

Notice that $\chi(S_1)$ in Example \ref{exmp: Computing Cluster Character} is the same as $\mu_1((x_1,x_2,x_3))$ from Example \ref{exmp: mutation of mutated seed}. This is indeed not a coincidence. Suppose we have a have a general quiver without oriented cycles. Then at vertex $k$, $Q$ and $Q^{op}$ appear locally as depicted on the left and right of the following diagram.

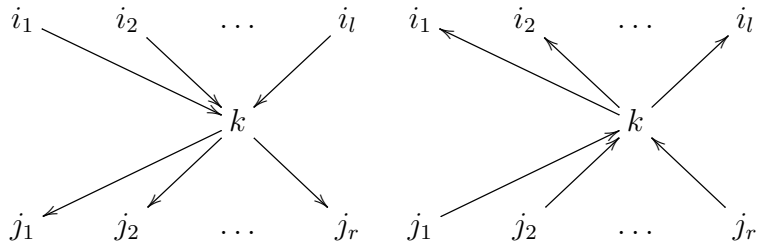
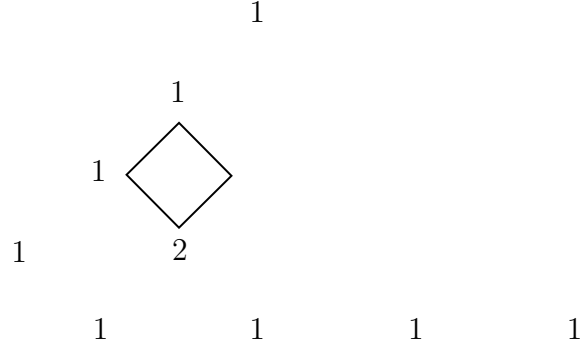
\begin{center}
\begin{tabular}{c c}
\xymatrix{            i_1 \ar[drr]                & i_2 \ar[dr]  &         \dots                           & i_l \ar [dl]\\
				                                           &                    &    k  \ar[dll]   \ar[dl]   \ar[dr]                             &       \\
						   j_1                & j_2    &          \dots                  & j_r } 
 & 
\xymatrix{            i_1                & i_2   &         \dots                           & i_l \\
				                                           &                    &    k  \ar[ull]   \ar[ul]   \ar[ur]                             &       \\
						   j_1  \ar[urr]               & j_2 \ar[ur]    &          \dots                  & j_r \ar [ul] } 
\end{tabular}
\end{center}

We attain a minimal projective $\Bbbk Q$ resolution of $S_k$ by $P_{j_1}\oplus P_{j_2}\oplus \dots \oplus P_{j_r} \hookrightarrow P_k \twoheadrightarrow S_k$ and a minimal projective $\Bbbk Q^{op}$ resolution of $D(S_k)$ given by $P_{i_1}\oplus P_{i_2}\oplus \dots \oplus P_{i_l} \hookrightarrow P_k \twoheadrightarrow S_k$. Therefore $x^{-g(D(S_k))} = {\prod x_i\over x_k}$ and $x^{-g(S_k)} = {\prod x_j\over x_k}$. Finally, we attain a formula for $\chi(S_k)$ given by
\vspace{-.25cm}
$$ \chi(S_k) = \chi(S_k,0) + \chi(S_k,S_k) = {\displaystyle \prod_{i\rightarrow k} x_i\over x_k} + {\displaystyle \prod_{k\rightarrow j} x_j\over x_k} = {\displaystyle \prod_{i\rightarrow k} x_i + \displaystyle \prod_{k\rightarrow j} x_j\over x_k}.$$

But this is precisely the formula for mutation at vertex $k$ given by Fomin and Zelevinsky in Definition \ref{defn: mutation of seed}! We conclude that mutation at vertex $k$ of the seed $(x_1, x_2, \dots , x_n)$ is precisely the cluster character of the simple at vertex $k$. As we have seen from the previous examples, the cluster characters seem to count the number of submodules. This is the case for modules that have finitely many submodules. The following theorem, which follows from Theorem \ref{thm: CC equation}, provides this result.

\begin{thm}\label{thm: what does the cluster character say finite case}
By setting all $x_i = 1$, we get that $\chi(M)|_{(1,1,\dots,1)}$ is precisely the number of submodules of $M$ so long as $M$ has finitely many submodules.
\end{thm} 

Using this theorem, we can can reconstruct the AR quiver for $Q$ when all $\Bbbk Q$-modules have finitely many submodules. Suppose the module $B = B_1 \oplus B_2$. By evaluating the cluster character at 1 and using the relationship $\chi(A)\chi(C) = \chi(B)+1$ for almost split sequences $A\hookrightarrow B \twoheadrightarrow C$, we have that the information of being an almost split sequence is encoded in the fact that the determinant of the following matrix is one $$\left[ \begin{array}{cc} \chi(A) & \chi(B_1) \\ \chi(B_2) & \chi(C) \\ \end{array} \right].$$ Moreover note that $\chi(M)=0$ if and only if $M = 0$.

\begin{exmp}\label{exmp: reconstructing the ar quiver using number of submodules}
Let $Q$ be the quiver $1\leftarrow 2\leftarrow3$. We begin with the array 

\begin{center}

\tikzset{every picture/.style={line width=0.75pt}} 

\begin{tikzpicture}[x=0.75pt,y=0.75pt,yscale=-1,xscale=1]

\draw   (113.46,172.16) -- (139.92,145.97) -- (166.37,172.69) -- (139.91,198.88) -- cycle ;

\draw (94,163.4) node [anchor=north west][inner sep=0.75pt]    {$1$};
\draw (134,122.4) node [anchor=north west][inner sep=0.75pt]    {$1$};
\draw (174,82.4) node [anchor=north west][inner sep=0.75pt]    {$1$};
\draw (214,42.4) node [anchor=north west][inner sep=0.75pt]    {$1$};
\draw (135,202.4) node [anchor=north west][inner sep=0.75pt]    {$1$};
\draw (294,202.4) node [anchor=north west][inner sep=0.75pt]    {$1$};
\draw (214,202.4) node [anchor=north west][inner sep=0.75pt]    {$1$};
\draw (374,202.4) node [anchor=north west][inner sep=0.75pt]    {$1$};

\end{tikzpicture}

\end{center}

The first rhombus indicates the first matrix we must complete, namely, $\left[ \begin{array}{cc} 1 & 1 \\ 1 &  \\ \end{array} \right].$ Solving for the missing entry to ensure this matrix has determinant one gives the first missing entry.

\begin{center}

\tikzset{every picture/.style={line width=0.75pt}} 

\begin{tikzpicture}[x=0.75pt,y=0.75pt,yscale=-1,xscale=1]

\draw   (153.46,131.16) -- (179.92,104.97) -- (206.37,131.69) -- (179.91,157.88) -- cycle ;

\draw (94,163.4) node [anchor=north west][inner sep=0.75pt]    {$1$};
\draw (134,122.4) node [anchor=north west][inner sep=0.75pt]    {$1$};
\draw (174,82.4) node [anchor=north west][inner sep=0.75pt]    {$1$};
\draw (214,42.4) node [anchor=north west][inner sep=0.75pt]    {$1$};
\draw (135,202.4) node [anchor=north west][inner sep=0.75pt]    {$1$};
\draw (294,202.4) node [anchor=north west][inner sep=0.75pt]    {$1$};
\draw (214,202.4) node [anchor=north west][inner sep=0.75pt]    {$1$};
\draw (374,202.4) node [anchor=north west][inner sep=0.75pt]    {$1$};
\draw (175,162.4) node [anchor=north west][inner sep=0.75pt]    {$2$};

\end{tikzpicture}

\end{center}
Continuing to find the missing entry in the matrix given by the rhombi will give the following picture.

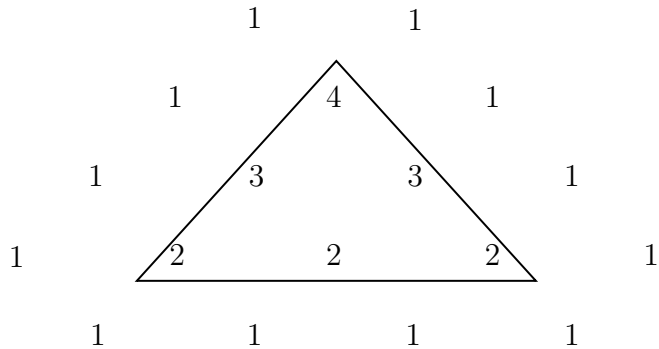
\begin{center}

\tikzset{every picture/.style={line width=0.75pt}} 

\begin{tikzpicture}[x=0.75pt,y=0.75pt,yscale=-1,xscale=1]

\draw   (260.61,71.01) -- (361.22,182) -- (160,182) -- cycle ;

\draw (94,163.4) node [anchor=north west][inner sep=0.75pt]    {$1$};
\draw (134,122.4) node [anchor=north west][inner sep=0.75pt]    {$1$};
\draw (174,82.4) node [anchor=north west][inner sep=0.75pt]    {$1$};
\draw (214,42.4) node [anchor=north west][inner sep=0.75pt]    {$1$};
\draw (135,202.4) node [anchor=north west][inner sep=0.75pt]    {$1$};
\draw (294,202.4) node [anchor=north west][inner sep=0.75pt]    {$1$};
\draw (214,202.4) node [anchor=north west][inner sep=0.75pt]    {$1$};
\draw (374,202.4) node [anchor=north west][inner sep=0.75pt]    {$1$};
\draw (175,162.4) node [anchor=north west][inner sep=0.75pt]    {$2$};
\draw (215,122.4) node [anchor=north west][inner sep=0.75pt]    {$3$};
\draw (254,82.4) node [anchor=north west][inner sep=0.75pt]    {$4$};
\draw (254,162.4) node [anchor=north west][inner sep=0.75pt]    {$2$};
\draw (295,122.4) node [anchor=north west][inner sep=0.75pt]    {$3$};
\draw (334,162.4) node [anchor=north west][inner sep=0.75pt]    {$2$};
\draw (295,43.4) node [anchor=north west][inner sep=0.75pt]    {$1$};
\draw (334,82.4) node [anchor=north west][inner sep=0.75pt]    {$1$};
\draw (374,122.4) node [anchor=north west][inner sep=0.75pt]    {$1$};
\draw (414,162.4) node [anchor=north west][inner sep=0.75pt]    {$1$};

\end{tikzpicture}

\end{center}

Inside the triangle is a depiction of the AR quiver for $Q$. Notice that each number provides the number of submodules of the corresponding module in the AR quiver: 

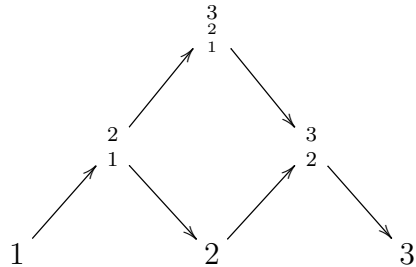
\begin{center}
\begin{tabular}{c}
\xymatrix{ & & {3\atop{2\atop 1}}\ar[dr] & & \\ & {2 \atop 1}\ar[ur] \ar[dr] & & {3\atop 2} \ar[dr] & \\ 1\ar[ur] & & 2\ar[ur] & &3 }
\end{tabular}
\end{center}
\end{exmp}

Note that in the statement of Theorem \ref{thm: what does the cluster character say finite case}, we restrict ourselves to the case in which $M$ has only finitely many submodules. This naturally raises the question, what does the cluster character of $M$ tell us when $M$ has infinitely many submodules? We proceed with an example.

\begin{exmp} \label{exmp: infinite submodules}
Let $Q$ be the quiver $1\leftarrow 2$. Then $\chi(S_2)|_1 = {x_1 + 1 \over x_2}\big|_1 = 2$. Consider the module $M = S_2 \oplus S_2$. Then $M$ has three types of submodules, namely $0, S_2,$ and $M$. There is only one submodule of the form $0$ and $M$; however, there are infinitely many of the form $S_2$. Given $(a,b) \in \Bbbk^2 - \{(0,0)\}$ we have an embedding $S_2 \hookrightarrow S_2\oplus S_2$ given by $x \mapsto (ax,bx)$. Note that scaling this embedding gives the same image. Therefore, we can realize the image of this embedding as the equivalence class of $(a,b) \in \Bbbk^2 - \{(0,0)\}$ up to scaling; that is, we can realize each submodule of $M$ of the form $S_2$ as an element of the projective line $\Bbbk P^1$. There are two methods to find the cluster character. 

\begin{enumerate}
\item We take $\Bbbk = \mathbb{C}$. Therefore the set of submodules of $M$ isomorphic to $S_2$ equals $\mathbb{C}P^1 \cong S^2$. We then set the cluster character $\chi(M,S_2)$ equal to the Euler characteristic of $\mathbb{C}P^1$. Recall the definition of Euler characteristic is $\chi(\mathbb{C}P^1) = \displaystyle \sum_k (-1)^k\text{dim}H_k(\mathbb{C}P^1) = 1 - 0 + 1$ where $H_k$ denotes the $k$th homology. There is a shortcut to compute this; namely, the Euler characteristic of a surface of genus $g$ is given by $\chi(\Sigma_g) = 2-2g$. In this case, our surface is of genus 0, so we attain the result. One thing to observe is that this number can be negative; however by Theorem \ref{Thm: cluster variables are Laurent Polynomials}, the coefficients in the Laurant polynomial are always positive. This provides us a restriction on the genus of this surface whenever the modules correspond to cluster characters!  

\item Let $\Bbbk = \mathbb{F}_q$ be the finite field with $q$ elements. The number of elements in $\Bbbk P^1$ is ${q^2-1\over q-1} = q+1$. One way to see this is through the fact that the points are given by $[1,a]$ for $a \in \Bbbk$ and the point at infinity $[0,1]$. By taking the field with one element, we get that $q = 1$ and that the number of elements in $\Bbbk P^1$ is 2. The reason why we get the same number as the Euler characteristic of $\mathbb{C}P^1$ follows from a deep theorem in number theory that is out of the scope of these notes.
\end{enumerate} 

\end{exmp}

As it turns out, the collection of submodules of fixed dimension of a given module has been thoroughly studied. 

\begin{defn}
Let $Q$ be a quiver and $M$ a $\Bbbk Q$-module. The space of all submodules of $M$ with dimension vector $\bm{e}$, denoted by Gr$(M,\bm{e})$, is called a \textbf{quiver grassmannian}.
\end{defn}

This is indeed a space, in fact, it is a projective variety. A special case is when $Q$ is just a point with no arrows. Then for the representation $M = \Bbbk ^n$, the quiver grassmannian Gr$(M,k)$ gives the ordinary grassmannian of $k$-planes in $n$-space. For more on quiver grassmannians, see Cerulli Irelli's lectures on quiver grassmannians [\ref{ref: CI quiver grassmannians}]. We have already seen another example of a quiver grassmannian:

\begin{exmp}
For the $M$ and $Q$ from the previous example, Gr$(M,(0,1)) \cong \Bbbk P^1$. Moreover, we've computed its Euler characteristic: $\chi(\text{Gr}(M,(0,1))) = 2$.
\end{exmp}

We can use the notion of quiver grassmannians to define the cluster character of any module $M$ as was done by Caldero and Chapoton in [\ref{ref: CC equation}] as follows.

\begin{thm}[Caldero-Chapoton] {\color{white} .}\label{thm: CC equation}\\
The cluster character of $M$ is given by 

$$\chi(M) = \displaystyle \sum_{\bm{e}} \chi(\text{Gr}(M,\bm{e}))x^{-g(\bm{e})}x^{-g(D(M/\bm{e}))}$$
where the sum is taken over all vectors $\bm{e}$ such that there is a submodule of $M$ of dimension vector $\bm{e}$.
\end{thm}

At this point, this definition may not seem well defined because it stipulates that $g$ vectors depend only on the dimension vectors of modules and not the actual module itself. This is indeed the case since we can also define the $g$ vector of a $\Bbbk Q$- module $V$ by $g(V) := C^{-1}_{\Bbbk Q}\cdot\text{dim}V$ where $C_{\Bbbk Q}$ is the \textbf{Cartan matrix} of $Q$, which is defined as the matrix whose $i$th column is the dimension vector of the projective representation at vertex $i$. This gives the following lemma.

\begin{lem}
The $g$ vector $g(V)$ depends only on the dimension vector $\bm{e} = \text{dim}V$.
\end{lem} 

Notice that in Example \ref{exmp: reconstructing the ar quiver using number of submodules}, we used the fact that $\chi(B_1 \oplus B_2) = \chi(B_1)\chi(B_2)$. We now have the necessary tools to provide a representation theoretic proof of this fact.

\begin{thm}\label{thm: cluster character of sum is product of characters}
For two $\Bbbk Q$ modules $A$ and $B$, the cluster character satisfies the following equation $\chi(A \oplus B) = \chi(A)\chi(B)$.
\end{thm}

\begin{proof}[Sketch of Proof]
It suffices to prove this when we set the $x_i = 1$. We begin by providing a correspondence between submodules of $A \oplus B$ and pairs of submodules of the form $(X,Y)$ where $X \leq A$ and $Y\leq B$. Define $\psi((X,Y)) = X \oplus Y$. Then this is an injection into the collection of submodules of $A \oplus B$. Now, we have a split exact sequence $A\hookrightarrow A\oplus B \overset{p_B}{\twoheadrightarrow} B$ and given any submodule $V \leq A\oplus B$, we have another exact sequence $V \cap A \hookrightarrow V \twoheadrightarrow P_B(V)$. We define a map $\phi(V) = (V\cap A, p_B(V))$. Note that $\phi \circ \psi = 1$, so $\phi$ is surjective. 

Now $\phi$ is not necessarily injective; however, we will show that it is injective on the level of Euler characteristics. Consider the pair of submodules $(X,Y)$. Then $\phi^{-1}((X,Y)) = \{ V \subset A \oplus B : V \cap A = X \text{ and } p_B(V) = Y\}$. By the Noether isomorphism theorem, this set is in bijection with $\{W \subset A/X \oplus B : W \cap A/X = 0 \text{ and } p_b(W) = Y\}$ via the correspondence $V \mapsto V/X$. We can uniquely realize all submodules $W$ in this latter set as the graph of a function from $Y$ to $A/X$:

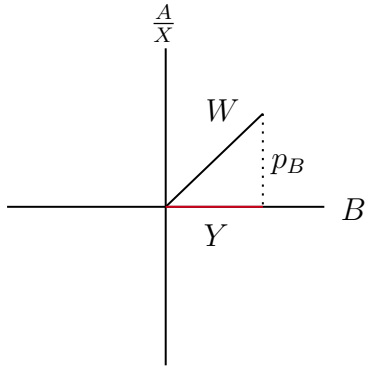
\begin{center}

\tikzset{every picture/.style={line width=0.75pt}} 

\begin{tikzpicture}[x=0.75pt,y=0.75pt,yscale=-1,xscale=1]

\draw    (330.22,96.01) -- (330.22,256.01) ;
\draw    (250.22,176.01) -- (410.22,176.01) ;
\draw    (330.22,176.01) -- (379.22,129.01) ;
\draw  [dash pattern={on 0.84pt off 2.51pt}]  (379.22,129.01) -- (379.22,176.01) ;
\draw [color={rgb, 255:red, 208; green, 2; blue, 27 }  ,draw opacity=1 ]   (330.22,176.01) -- (379.22,176.01) ;

\draw (349,120.4) node [anchor=north west][inner sep=0.75pt]    {$W$};
\draw (321,72.4) node [anchor=north west][inner sep=0.75pt]    {$\frac{A}{X}$};
\draw (417,170.4) node [anchor=north west][inner sep=0.75pt]    {$B$};
\draw (348,183.4) node [anchor=north west][inner sep=0.75pt]    {$Y$};
\draw (382,148) node [anchor=north west][inner sep=0.75pt]    {$p_B$};

\end{tikzpicture}

\end{center}

In particular, with a little more effort, we have that $\chi(\phi^{-1}(X,Y)) = \chi(\text{Hom}(Y,A/X))$. When we consider $\text{Hom}(Y,A/X)$ over a field with $q$ elements, $\chi(\text{Hom}(Y,A/X)) = q^l$ for some $l$. By taking $q = 1$, we have that $\chi(\text{Hom}(Y,A/X)) = 1$ and therefore, $\chi(\text{Gr}(A \oplus B, \bm{e})) = \displaystyle \sum_{\bm{e_1} + \bm{e_2} = \bm{e}} \chi(\text{Gr}(A,\bm{e_1}) \times \text{Gr}(B, \bm{e}_2)) = \displaystyle \sum_{\bm{e_1} + \bm{e_2} = \bm{e}} \chi(\text{Gr}(A,\bm{e_1}) \chi(\text{Gr}(B, \bm{e}_2))$. 

Finally, since $g$-vectors are linear; that is, $g(V) = g(X) + g(Y)$, we have that $g(D(A\oplus B)/X) = g(D(A/X)) + g(D(B/Y))$. By summing over all possible dimension vectors $\bm{e}$, we have the result.
\end{proof}

Recall from Section \ref{sec: Cluster Variables} that when we have an almost spit sequence $A \hookrightarrow B \twoheadrightarrow C$, the cluster characters satisfy the equation $\chi(A)\chi(C) = \chi(B) + 1$ where $\chi(B) = \chi(\oplus_i B_i) = \sum_i \chi(B_i)$. We now have the tools to prove this.

\begin{thm}\label{thm: computing cluster characters for projectives}
For an almost split sequence of $\Bbbk Q$-modules $A \overset{q}{\hookrightarrow} B \overset{p}{\twoheadrightarrow} C$, we have $$\chi(A\oplus C) = \chi(A)\chi(C) = \chi(B) + 1.$$
\end{thm}

\begin{proof}[Sketch of Proof]
We proceed as in the proof of the previous theorem. Consider the collection of tuples $(X,Y)$ where $X \leq A$ and $Y\leq B$ are submodules, and let $V\leq B$ be a submodule. We define a map $\phi: \{\text{submodules of $B$}\} \rightarrow \{\text{pairs $(X,Y)$}\}$ by $\phi(V) = (V\cap A, p(V))$. Consider a tuple $(X,Y)$ such that $Y \neq 0$. Then since $p$ is irreducible and the inclusion of $Y$ into $C$ is not a retraction, we have the following commutative diagram.

\begin{center}
\begin{tabular}{c}
\xymatrix{X\ar^{q \circ i}[dr] \ar_i[d] & & Y \ar@{.>}_f[dl] \ar^i[d] \\ A\ar^q[r] & B\ar^p[r] & C}
\end{tabular}
\end{center} 

Define $V = q(X) + f(Y)$, so that $\phi(V) = (X,Y)$ and $(X,Y)$ is in the image of $\phi$. Now suppose that $X \neq 0$. Then since the quotient map is not a section, we have the following pushout diagram.

\begin{center}
\begin{tabular}{c}
\xymatrix{A \ar[d] \ar^q[r] & B \ar^p[r] \ar@{.>}_f[dl] \ar^{\pi}[d] & C  \ar^{\text{id}}[d] \\ A/X \ar[r] & A/X \oplus C \ar[r] & C}
\end{tabular}
\end{center} 

Define $V = \pi^{-1}((0,Y)) + X$, thus $\phi(V) = (X,Y)$ and we conclude that $(X,Y)$ is in the image of $\phi$. Therefore the map $\phi$ is onto $\{\text{pairs $(X,Y)$}\} - \{(0,C)\}$ since $C$ can't be lifted due to the irreducibilty of $p$. With some more work, we conclude that $\chi(B) = \chi(A)\chi(C) - 1$ where the minus one comes from the fact that the submodule $(0,C)$ can't be lifted.
\end{proof}

We have techniques to compute cluster characters of modules that lie in an almost split sequence; however, we do not have any techniques to compute cluster characters of modules that do not lie at the end of an almost split sequence. Although we know how to use the definition to compute cluster characters of any $\Bbbk Q$-module, we will now provide a theorem which will provide a technique to compute the cluster characters of projectives. 

\begin{thm}
Let $P_i$ denote the projective $\Bbbk Q$-module at vertex $i$. Then we have $$\chi(P_i) = \chi(\text{rad}P_i)x^{-g(DS_i)} + x_i^{-1}$$ where $S_i$ is the simple top of $P_i$.
\end{thm}

We will not prove this, but instead provide an example. As the reader may have guessed, the proof is given by Caldero and Chapoton in [\ref{ref: CC equation}].

\begin{exmp}
Let $Q$ be the quiver $1 \leftarrow 2 \leftarrow 3$ so that $Q^{op}$ is $1\rightarrow 2 \rightarrow 3$. We compute $\chi(P_3)$ where $P_3 = {3 \atop{2\atop1}}$, $\text{rad}P_3 = {2\atop1} = P_2$, and $\text{rad}P_2 = 1 = P_1 = S_1$. Thus $$\chi(P_3) = \chi(P_2)x^{-g(DS_3)} + x_3^{-1} = (\chi(S_1)x^{-g(DS_2)} + x_2^{-1})x^{-g(DS_3)} + x_3^{-1}.$$

We compute $$\chi(S_1) = x^{-g(0)}x^{-g(D(S_1))} + x^{-g(S_1)}x^{-g(0)} = x^{(-1,1,0)} + x^{(-1,0,0)} = {x_2 + 1 \over x_1}.$$

Now we need $\chi(P_2) = \chi(S_1)x^{-g(DS_2)} + x_2^{-1}$. We have $$\chi(P_2) = {x_2 + 1 \over x_1}x^{(0,-1,1)} + x_2^{-1} = {x_2x_3 + x_3 + x_1 \over x_1x_2}.$$

Finally, we compute $$\chi(P_3) = \chi(P_2)x^{-g(DS_3)} + x_3^{-1} = {x_2x_3 + x_3 + x_1 \over x_1x_2}x^{(0,0,-1)} + x_3^{-1} = {x_2x_3 + x_3 + x_1 + x_1x_2 \over x_1x_2}.$$
\end{exmp}

In the case in which the AR quiver of our algebra is connected, we can get the cluster character of injectives simply by completing the AR quiver. However, in the case in which the AR quiver is disconnected, we can't get the cluster character of injectives by beginning at projectives and completing meshes. For instance in the tame case, we would never leave the preprojective component. By duality, we have a dual theorem to Theorem \ref{thm: computing cluster characters for projectives} that provides us with a technique to compute the cluster character of injectives. 

\begin{thm}
Let $I_i$ denote the injective $\Bbbk Q$-module at vertex $i$. Then we have $\chi(I_i) = \chi(I_i/S_i)x^{-g(S_i)} + x_i^{-1}$.
\end{thm}

\begin{exmp}
Let's take $Q$ to be the Kronecker $1 \leftleftarrows 2$ and compute $\chi(I_1)  = \chi({2 \, 2 \atop 1})$. By the previous theorem, we compute $\chi(S_2 \oplus S_2)x^{-g(S_1)} + x_1^{-1}$. Note by Theorem \ref{thm: cluster character of sum is product of characters}, we have that $\chi(S_2 \oplus S_2) = \chi(S_2)\chi(S_2)$. We begin by computing $\chi(S_2) = {x_1^2 + 1 \over x_2}$. Since $g(S_1) = (1,0)$, we have that $$\chi(I_1) = \bigg({x_1^2 + 1 \over x_2}\bigg)\bigg({x_1^2 + 1 \over x_2}\bigg)x^{(-1,0)} + {1\over x_1} = {x_1^4 + 2x_1^2 + x_2^2 + 1 \over x_2^2x_1}.$$

Notice that this is a case analogous to Example \ref{exmp: infinite submodules} in which there are infinitely many submodules of $I_1$ of the form $S_2 \oplus S_2$. So in this case, the coefficient of $x_1^2$ in the cluster character provides the Euler characteristic of Gr$(I_1,(1,1))$.
\end{exmp}

\section{Clusters with Modules}\label{sec: clusters with modules}

\indent

Although not yet explicitly stated, thus far we have seen that the cluster character $\chi$ sends indecomposable rigid modules to cluster variables, where by \textbf{rigid} we mean Ext$^1(M,M) = 0$. The two questions we wish to answer in this section are 

\begin{itemize}
\item Which sets of modules are sent to clusters?
\item Given a rigid indecomposable module $M$, we get $\chi(M) = { \text{something} \over x^{\text{dim}M}}$. Which algebraic objects correspond cluster variables of the form $x_i$ and what object corresponds to the initial cluster? 
\end{itemize}

To do this, we need to introduce and establish some more algebraic machinery. For more on classical tilting theory, see [\ref{ref: blue book}].

\begin{defn}
Let $\Lambda = \Bbbk Q$ where $Q$ is a quiver with $n$ vertices and no oriented cycles. Then a \textbf{tilting module} is a rigid module $T$ with $n$ indecomposable non-isomorphic summands $T = T_1 \oplus T_2 \oplus \dots \oplus T_n$. 
\end{defn}

The condition that $T$ is rigid implies that Ext$^1(T,T) = \displaystyle \oplus_{i,j} $Ext$^1(T_i,T_j) = 0$. This happens if and only if each $T_i$ is rigid and they don't extend each other, that is, Ext$^1(T_i,T_j) = 0$ for all $i$ and $j$.

\begin{exmp} Some examples of tilting modules are the following.
\begin{itemize}
\item The sum of the projective modules $\Lambda = P_1 \oplus P_2 \oplus \dots \oplus P_n$.
\item The sum of the injective modules $\nu\Lambda = I_1 \oplus I_2 \oplus \dots \oplus I_n$ where $\nu$ is the Nakayama functor.
\item Let $Q$ be the quiver $1 \leftarrow 2$. Then AR quiver is as follows.
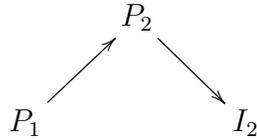
\begin{center}
\begin{tabular}{c}
\xymatrix{  & P_2 \ar[dr] \\ P_1 \ar[ur] & & I_2} 
\end{tabular}
\end{center}
Then there are 2 tilting $\Bbbk Q$-modules given by $S_1 \oplus P_2$ and $P_2 \oplus S_2$. Note that $S_1 \oplus S_2$ is not tilting since there is an extension of $S_2$ by $S_1$.

\item Let $Q$ be the quiver $1 \leftarrow 2 \leftarrow 3$. Then the AR quiver is as follows.

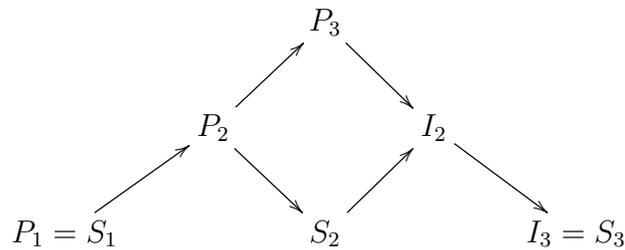
\begin{center}
\begin{tabular}{c}
\xymatrix{  & & P_3 \ar[dr] & & \\ & P_2 \ar[ur] \ar[dr]& & I_2\ar[dr] & \\ P_1 =S_1 \ar[ur] & & S_2\ar[ur] & & I_3 = S_3} 
\end{tabular}
\end{center}

To find all 5 tilting $\Bbbk Q$-modules, we must look for sets of three modules in the AR quiver that do not form a mesh, or almost split sequence. To do this, we can draw the \textbf{compatibility graph}, which is the graph in which there is an edge between any two indecomposable modules that form a rigid pair. This idea first originated as generalized associahedra (Stasheff polytopes) by Chapoton, Fomin, and Zelevinsky in [\ref{ref: CFZ original compatibility graph}]. Later in [\ref{ref: MRZ compatibility graph}], Marsh, Reineke, and Zelevinsky constructed these generalized associahedra using the representation theory of quivers and something called the category of decorated representations. Below is the compatibility graph for this example.

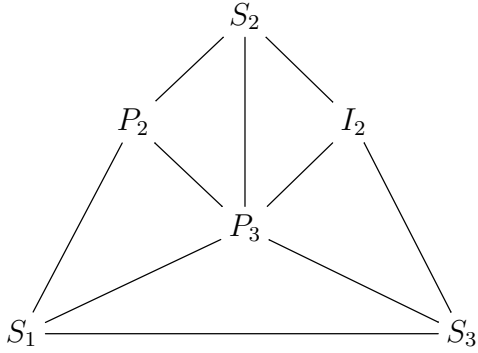
\begin{center}
\begin{tabular}{c}
\xymatrix{  & & S_2 \ar@{-}[dr] & & \\ & P_2 \ar@{-}[ur] \ar@{-}[dr]& & I_2\ar@{-}[ddr] & \\ & & P_3\ar@{-}[ur] \ar@{-}[uu]& & \\ S_1 \ar@{-}[uur] \ar@{-}[rrrr] \ar@{-}[urr] & &  & & S_3 \ar@{-}[ull]} 
\end{tabular}
\end{center}

The 5 tilting $\Bbbk Q$-modules are given by the triangles in the compatibility graph. We can re-write this using dimension vectors as follows: 

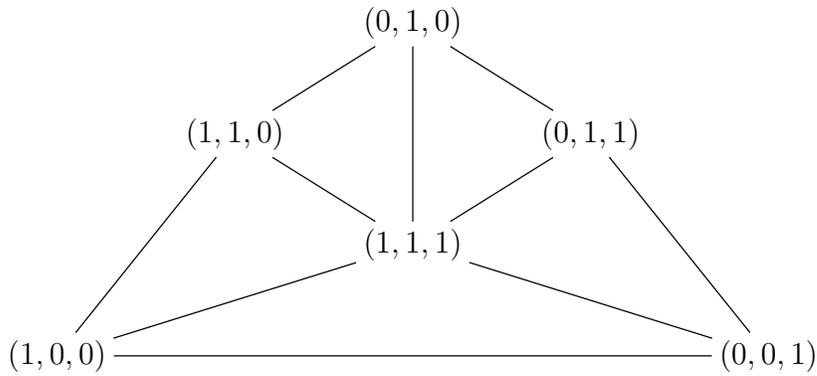
\begin{center}
\begin{tabular}{c}
\xymatrix{  & & (0,1,0) \ar@{-}[dr] & & \\ & (1,1,0) \ar@{-}[ur] \ar@{-}[dr]& & (0,1,1)\ar@{-}[ddr] & \\ & & (1,1,1)\ar@{-}[ur] \ar@{-}[uu]& & \\ (1,0,0) \ar@{-}[uur] \ar@{-}[rrrr] \ar@{-}[urr] & &  & & (0,0,1) \ar@{-}[ull]} 
\end{tabular}
\end{center}

Notice that any dimension vector that is a linear combination of any two other dimension vectors lies on the line connecting the two dimension vectors.

\end{itemize}
\end{exmp}

\subsection{Cluster Variables}

\indent

A cluster algebra whose initial quiver is of type $A_3$ has 9 cluster variables, 6 from the cluster character of the 6 indecomposable rigid modules and 3 from the initial cluster $(x_1,x_2,x_3)$. We can put all these cluster variables in an \textbf{extended compatibility graph} where we draw an edge between any two clusters variables that occur in a cluster.
\vspace{1cm}
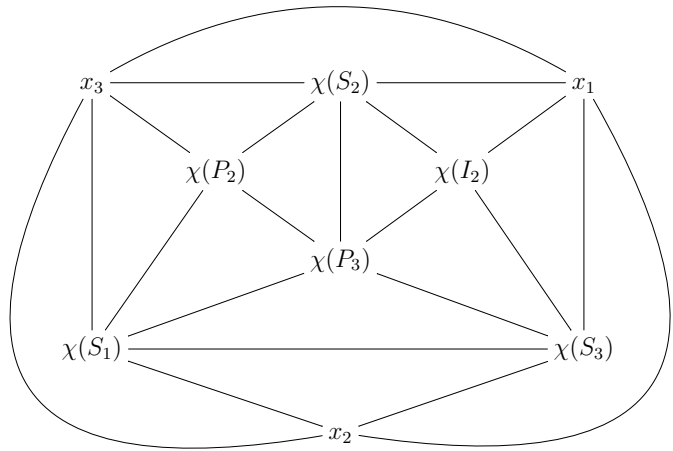
\begin{center}
\scalebox{0.8}{%
\begin{tabular}{c}
\xymatrix{ x_3 \ar@{-}@/_9pc/[ddddrr] \ar@{-}@/^3pc/[rrrr]  \ar@{-}[rr] \ar@{-}[dr] \ar@{-}[ddd] & & \chi(S_2) \ar@{-}[dr] & & x_1 \ar@{-}@/^9pc/[ddddll] \ar@{-}[ll] \ar@{-}[dl] \ar@{-}[ddd]\\ & \chi(P_2) \ar@{-}[ur] \ar@{-}[dr]& & \chi(I_2)\ar@{-}[ddr] & \\ & & \chi(P_3)\ar@{-}[ur] \ar@{-}[uu]& & \\ \chi(S_1) \ar@{-}[uur] \ar@{-}[rrrr] \ar@{-}[urr] & &  & & \chi(S_3) \ar@{-}[ull] \\ & & x_2 \ar@{-}[ull] \ar@{-}[urr]& & } 
\end{tabular}}
\end{center}

Including the trivial outer triangle that corresponds to the initial cluster, there are 14 triangles in the compatibility graph. We can visualize cluster mutation as `wall crossing' in the compatibility graph and moreover, from the fact that the compatibility graph is a manifold, we see that we can get between any two clusters, or triangles, by a sequence of mutations. Consider the sequence of mutations of the initial cluster $(x_1,x_2,x_3) \overset{\mu_1}{\rightarrow} (\chi(x_1),x_2,x_3) \overset{\mu_3}{\rightarrow} (\chi(x_1),x_2,\chi(x_3))$. Then this is seen in the compatibility graph as the following sequence of `wall crossings' where the corresponding triangle has edges given by the cluster variables in the mutated cluster.

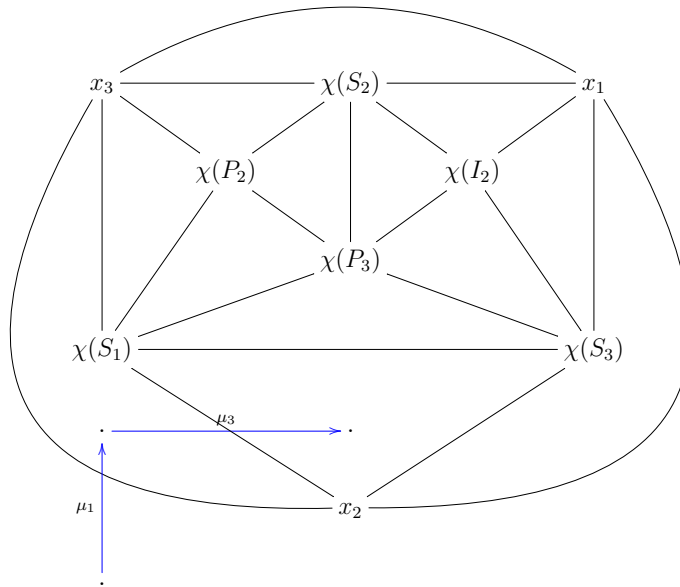
\begin{center}
\scalebox{0.8}{%
\begin{tabular}{c}
\xymatrix{ x_3 \ar@{-}@/_9pc/[dddddrr] \ar@{-}@/^3pc/[rrrr]  \ar@{-}[rr] \ar@{-}[dr] \ar@{-}[ddd] & & \chi(S_2) \ar@{-}[dr] & & x_1 \ar@{-}@/^9pc/[dddddll] \ar@{-}[ll] \ar@{-}[dl] \ar@{-}[ddd]\\ & \chi(P_2) \ar@{-}[ur] \ar@{-}[dr]& & \chi(I_2)\ar@{-}[ddr] & \\ & & \chi(P_3)\ar@{-}[ur] \ar@{-}[uu]& & \\ \chi(S_1) \ar@{-}[uur] \ar@{-}[rrrr] \ar@{-}[urr] & &   & & \chi(S_3) \ar@{-}[ull] \\ \cdot \ar@[blue]^{\mu_3}[rr] & & \cdot & & \\ & & x_2 \ar@{-}[uull] \ar@{-}[uurr]& & \\ \cdot \ar@[blue]^{\mu_1}[uu] & & & & } 
\end{tabular}}
\end{center}

\subsection{Support Tilting Modules and Shifted Projectives}

\indent 

We will now try to understand why there are 14 clusters, but only 5 tilting modules. To do this, we need to weaken the notion of a tilting module to a support tilting module which was introduced by Ingalls and Thomas in [\ref{ref: IT support tilting}].

\begin{defn}
A $\Bbbk Q$ module $T = T_1 \oplus T_2 \oplus \dots \oplus T_k$ where $k \leq n$ is called \textbf{support tilting} if the following hold.
\begin{enumerate}
\item $T$ is rigid.
\item The $T_i$ are non-isomorphic.
\item The support of $T$ has $k$ elements.
\end{enumerate}
\end{defn}

It is condition 3 that motivates the name support tilting because this condition stipulates that the module is tilting on its support. Recall that the support of a module is the set of vertices of the quiver at which the corresponding representation has a non-zero vector space. 

\begin{exmp}
Take $Q$ to be $1\leftarrow 2\leftarrow 3$. 
\begin{itemize}
\item $P_1 \oplus P_2$ is support tilting.
\item Any $S_i$ is support tilting.
\item $P_2 \oplus P_3$ is \textbf{not} support tilting. Note that $|$supp$(P_2 \oplus P_3)| = |\{1,2,3\}| = 3 \neq 2$.  
\end{itemize}
\end{exmp}

We will now introduce the algebraic objects that correspond to the initial cluster variables. A \textbf{shifted projective} $P_i[1]$ is an object with projective presentation $P_i \rightarrow 0 \rightarrow P_i[1]$. These are more naturally realized in the bounded derived category of mod-$\Lambda$ where the shift $[1]$ denotes the shift functor in $D^b(\text{mod-}\Lambda)$ and the presentation is a distinguished triangle in the category. After applying the Nakayama functor to the aforementioned projective presentation, we get $0\rightarrow I_i =\tau P_i[1] \rightarrow I_i \rightarrow 0$, so we conclude that $\tau(P_i[1]) = I_i$. For $Q : 1 \leftarrow 2\leftarrow 3$, the AR quiver with the shifted projectives is the following:

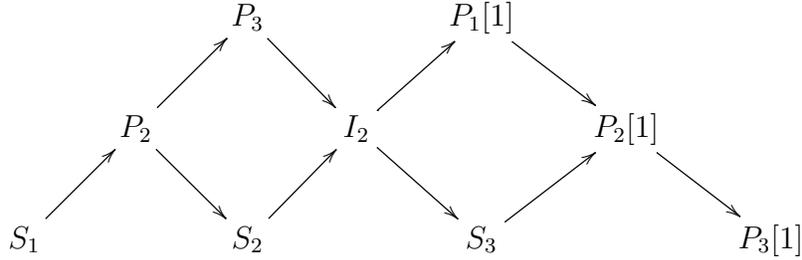
\begin{center}
\begin{tabular}{c}
\xymatrix{  & & P_3 \ar[dr] & &P_1[1] \ar[dr] & & \\ & P_2 \ar[ur] \ar[dr]& & I_2\ar[dr] \ar[ur] & & P_2[1] \ar[dr]& \\ S_1 \ar[ur] & & S_2\ar[ur] & & S_3\ar[ur] & & P_3[1] } 
\end{tabular}
\end{center}

Really what we are looking at here is the `cluster category' intoduced by Buan, Marsh, Reineke, Reiten, and Todorov in [\ref{ref: BMRRT categorification of cluster algebras}].

\subsection{Support Tilting Pairs and Extending the Cluster Character}

\indent 

The notion of support tilting (silting) pairs was first introduced in [\ref{ref: BPP silting pairs}] by Broomhead, Pauksztello and Ploog, though as we will see, the notion of silting objects is older.

\begin{defn}
A \textbf{silting pair} is a tuple of $\Bbbk Q$-modules $(T,P)$ where $T = T_1 \oplus T_2 \oplus \dots \oplus T_k$ is a support tilting module and $P = P_{j_1} \oplus P_{j_2} \oplus \dots \oplus P_{j_{n-k}}$ is a projective module with $n-k$ components whose simple tops $S_{j_1}, \dots , S_{j_{n-k}}$ at the vertices of $Q$ are \textit{not} in the support of $T$.
\end{defn}

Note the condition on the tops of the projective summands is equivalent to Hom$(P,T) = 0$. 

\begin{exmp}Taking $Q$ to be the linear $\mathbb{A}_3$ quiver from above, we have the following silting pairs.
\begin{itemize}
\item $(P_1 \oplus P_3, P_3)$ is a silting pair.
\item $(S_1, P_2 \oplus P_3)$ is a silting pair.
\end{itemize}
\end{exmp}

We are now ready to completely explain the compatibility diagram in terms of algebraic objects. The first part of the bijection in the next theorem was first proven by Buan, Marsh, Reineke, Reiten, and Todorov in [\ref{ref: BMRRT categorification of cluster algebras}]. They moreover showed that the tilting objects in the cluster category, which are formed by taking direct sums of indecomposables and shifted projectives, are in bijection with clusters. Adachi, Iyama, and Reiten in [\ref{ref: AIR tau-tilting}] showed that silting pairs $(T_1 \oplus T_2 \oplus \dots \oplus T_k,P_{j_1} \oplus P_{j_2} \oplus \dots \oplus P_{j_{n-k}})$ and the aforementioned tilting objects are in bijection via the map $(T,P) \mapsto  T_1 \oplus T_2 \oplus \dots \oplus T_k \oplus P_{j_1}[1] \oplus P_{j_2}[1] \oplus \dots \oplus P_{j_{n-k}}[1]$. This is one way to get the latter part of the below bijection.  

\begin{thm}
There is a bijection between the collections $\{$rigid indecomposable $\Bbbk Q$-modules $M$ and shifted projectives$\}$ and $\{$cluster variables in the corresponding cluster algebra$\}$. This bijection induces a bijection $\varphi : \{ \text{silting pairs} \} \rightarrow \{ \text{clusters} \}$ by $\varphi((T,P)) = (\chi(T_1), \chi(T_2),\dots, \chi(T_k), x_{j_1}, \dots , x_{j_{n-k}})$.
\end{thm}

\begin{rem}
Recall in Example \ref{exmp: infinite submodules}, we mentioned that the Euler characteristic of the quiver grassmannian of any module can be computed using the shortcut $2-2g$ where $g$ is the genus of the quiver grassmannian. By the positivity conjecture, Theorem \ref{Thm: cluster variables are Laurent Polynomials}, this number is always nonnegative whenever $M$ corresponds to a cluster variable. The previous theorem allows us to conclude that the genus of the quiver grassmannian of a rigid module $M$ is at most 1.
\end{rem}

We can now define the \textbf{extended compatibility graph} in terms of representation theory, which was done by Buan, Marsh, Reineke, Reiten, and Todorov in [\ref{ref: BMRRT categorification of cluster algebras}], as follows. We connect $P_i[1]$ and $M$ with an edge if Hom$(P_i,M) = 0$, we connect $M$ and $N$ with an edge if Ext$(M,N) =$ Ext$(N,M) = 0$, and we connect $P_i[1]$ and $P_j[1]$ with an edge so long as $i \neq j$. The extended compatibility graph for the linear $\mathbb{A}_3$ quiver is as follows.

\vspace{1cm}
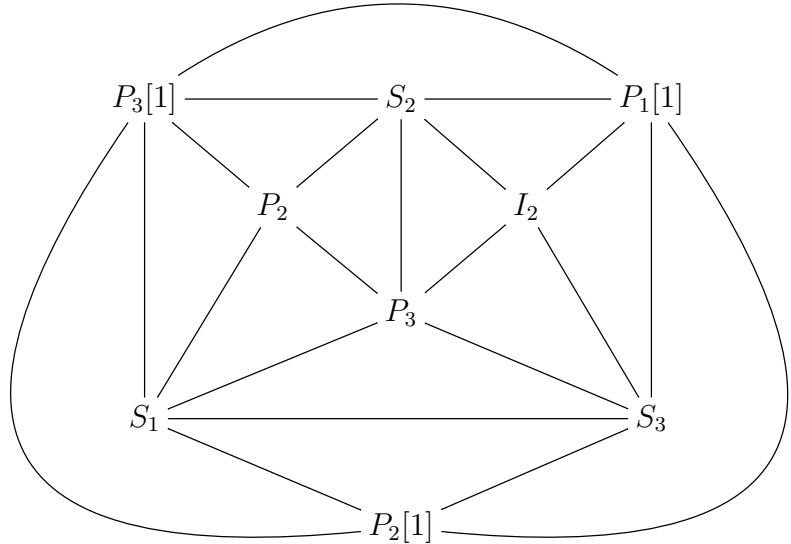
\begin{center}
\begin{tabular}{c}
\xymatrix{ P_3[1] \ar@{-}@/_9pc/[ddddrr] \ar@{-}@/^3pc/[rrrr]  \ar@{-}[rr] \ar@{-}[dr] \ar@{-}[ddd] & & S_2 \ar@{-}[dr] & & P_1[1] \ar@{-}@/^9pc/[ddddll] \ar@{-}[ll] \ar@{-}[dl] \ar@{-}[ddd]\\ & P_2 \ar@{-}[ur] \ar@{-}[dr]& & I_2 \ar@{-}[ddr] & \\ & & P_3 \ar@{-}[ur] \ar@{-}[uu]& & \\ S_1 \ar@{-}[uur] \ar@{-}[rrrr] \ar@{-}[urr] & &  & & S_3 \ar@{-}[ull] \\ & & P_2[1] \ar@{-}[ull] \ar@{-}[urr]& & } 
\end{tabular}
\end{center}

Notice that in the above extended compatibility graph, a nontrivial edge is only shared by at most two triangles. For instance if we consider the silting pair $(S_1 \oplus S_3,P_2)$, we are analyzing the triangle with vertices given by $S_1, S_3,$ and $P_2[1]$. We see that if we remove $P_2[1]$, there is only one other triangle that has two vertices given by $S_1$ and $S_3$, namely the triangle whose third vertex is $P_3$. This is an example of the following not algebraically obvious theorem proven by Buan, Marsh, Reineke, Reiten, and Todorov in [\ref{ref: BMRRT categorification of cluster algebras}].

\begin{thm}
Given a silting pair $(T,P)$ and one object in the cluster $T_i$ or $P_i[1]$, there is exactly one way to replace that object with another object such that the new collection of objects is a silting pair.
\end{thm}

\subsection{Stability Pictures}\label{sec: stability pictures}

\indent

The stability picture, more commonly known as the wall and chamber structure, is another method of studying quiver algebras. We will now define the stability picture and see how it is related to the extended cluster diagram. We first need the notion of stability conditions, which was first studied by King in [\ref{ref: K stability conditions}]. In [\ref{ref: B scattering diagrams}], Bridgeland used these stability conditions to construct a scattering diagram whose support is the so-called wall and chamber structure which we define below.

\begin{defn} 
For $\theta\in\mathbb{R}^n$, a non-zero module $M \in \text{mod}$-$\Bbbk Q$ is called \textbf{$\theta$-stable} if it is orthogonal to $\theta$, that is, $\theta\cdot M = \theta \cdot \text{dim}M = 0$, and $\theta \cdot L < 0 $ for every proper submodule $L$ of $M$. Moreover, a module $M$ orthogonal to $\theta$ is called \textbf{$\theta$-semistable} if $\theta \cdot L \leq 0$ for every submodule $L$ of $M$.
\end{defn}

We wish to study all the values of $\theta$ such that a fixed module $M$ is $\theta$-stable.

\begin{defn}
The stability space of a $\Bbbk Q$-module $M$ is $\mathcal{D}(M) = \{\theta \in \mathbb{R}^n :  M \, \text{is}\, \theta\text{-semistable}\}$. 
\end{defn}

When $\mathcal{D}(M)$ has codimension 1, we call it a \textbf{wall}. It is not the case that the stability space of every indecomposable module always gives a wall. Precisely which modules have walls as their stability space has been worked out by Treffinger in [\ref{ref: T which bricks correspond to $c$-vectors}].

\begin{defn}
Let $$\mathcal{R} = \mathbb{R}^n - \overline{\bigcup_{M\in \text{mod} \Bbbk Q} \mathcal{D}(M)}$$ denote the maximal open set of $\theta$ having no $\theta$-semistable non-zero modules. Then a connected component $\mathcal{C}$ of $\mathcal{R}$ is called a \textbf{chamber}.
\end{defn}

\begin{exmp}
Let $Q = 1\leftarrow 2$. Then we have three indecomposable modules $S_1, P_2,$ and $S_2$ with respective dimension vectors $(1,0), (1,1),$ and $(0,1)$. Given $\theta \in \mathbb{R}^2$, we have $(a,b)\cdot (1,0) = 0 \iff a = 0$, so $\mathcal{D}(S_1) = \{(a,b) : a = 0\}$. Similarly, $\mathcal{D}(S_2) = \{(a,b) : b = 0\}$. Finally, $(a,b)\cdot (1,1) = 0 \iff a = -b$. But note that since $S_1 \leq P_1$, we also require that  $(a,b)\cdot (1,0) \leq 0 \iff a \leq 0$. Therefore $\mathcal{D}(P_2) = \{(a,b) : a = -b \, \text{and} \, a \leq 0\}$. Below is a depiction of the wall and chamber structure, which we also call the \textbf{stability picture}. In this stability picture we also have the $g$-vectors of each indecomposable module and shifted projective. Note that the $g$-vector of the shifted projective $P_i[1]$ is $(0, \dots, 0, -1, 0, \dots, 0)$ where the $-1$ is in the $i$th postition, since any shifted projective has a projective presentation $P_i \rightarrow 0 \rightarrow P_i[1]$.

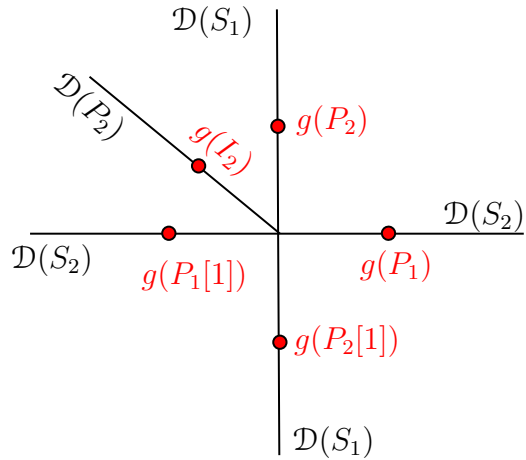
\begin{center}

\tikzset{every picture/.style={line width=0.75pt}} 

\begin{tikzpicture}[x=0.75pt,y=0.75pt,yscale=-1,xscale=1]

\draw    (318.72,18.01) -- (319.93,243.01) ;
\draw    (443.22,131.22) -- (194.22,131.22) ;
\draw    (320.28,131.22) -- (224.22,52.01) ;
\draw    (375.22,130.01) -- (375,131) ;
\draw [shift={(375,131)}, rotate = 102.55] [color={rgb, 255:red, 0; green, 0; blue, 0 }  ][fill={rgb, 255:red, 255; green, 0; blue, 0 }  ][line width=0.75]      (0, 0) circle [x radius= 3.35, y radius= 3.35]   ;
\draw    (319.22,78.01) -- (319.22,77.01) ;
\draw [shift={(319.22,77.01)}, rotate = 270] [color={rgb, 255:red, 0; green, 0; blue, 0 }  ][fill={rgb, 255:red, 255; green, 0; blue, 0 }  ][line width=0.75]      (0, 0) circle [x radius= 3.35, y radius= 3.35]   ;
\draw    (265.22,131.01) -- (264.22,131.01) ;
\draw [shift={(264.22,131.01)}, rotate = 180] [color={rgb, 255:red, 0; green, 0; blue, 0 }  ][fill={rgb, 255:red, 255; green, 0; blue, 0 }  ][line width=0.75]      (0, 0) circle [x radius= 3.35, y radius= 3.35]   ;
\draw    (320.22,186.01) ;
\draw [shift={(320.22,186.01)}, rotate = 0] [color={rgb, 255:red, 0; green, 0; blue, 0 }  ][fill={rgb, 255:red, 255; green, 0; blue, 0 }  ][line width=0.75]      (0, 0) circle [x radius= 3.35, y radius= 3.35]   ;
\draw    (280.22,97.01) -- (279.22,97.01) ;
\draw [shift={(279.22,97.01)}, rotate = 180] [color={rgb, 255:red, 0; green, 0; blue, 0 }  ][fill={rgb, 255:red, 255; green, 0; blue, 0 }  ][line width=0.75]      (0, 0) circle [x radius= 3.35, y radius= 3.35]   ;

\draw (264,15) node [anchor=north west][inner sep=0.75pt]   [align=left] {$\mathcal{D}(S_1)$};
\draw (325,227) node [anchor=north west][inner sep=0.75pt]   [align=left] {$\mathcal{D}(S_1)$};
\draw (183,135) node [anchor=north west][inner sep=0.75pt]   [align=left] {$\mathcal{D}(S_2)$};
\draw (401,113) node [anchor=north west][inner sep=0.75pt]   [align=left] {$\mathcal{D}(S_2)$};
\draw (212.67,45.7) node [anchor=north west][inner sep=0.75pt]  [rotate=-40.91] [align=left] {$\mathcal{D}(P_2)$};

\draw (359,138) node [anchor=north west][inner sep=0.75pt]  [color={rgb, 255:red, 255; green, 0; blue, 0  }  ,opacity=1 ] [align=left] {$\displaystyle g( P_{1})$};
\draw (327,65) node [anchor=north west][inner sep=0.75pt]  [color={rgb, 255:red, 255; green, 0; blue, 0 }  ,opacity=1 ] [align=left] {$\displaystyle g( P_{2})$};
\draw (326,176) node [anchor=north west][inner sep=0.75pt]  [color={rgb, 255:red, 255; green, 0; blue, 0 }  ,opacity=1 ] [align=left] {$\displaystyle g( P_{2}[ 1])$};
\draw (249,143) node [anchor=north west][inner sep=0.75pt]  [color={rgb, 255:red, 255; green, 0; blue, 0 }  ,opacity=1 ] [align=left] {$\displaystyle g( P_{1}[ 1])$};
\draw (283.41,69.28) node [anchor=north west][inner sep=0.75pt]  [color={rgb, 255:red, 255; green, 0; blue, 0  }  ,opacity=1 ,rotate=-39.84] [align=left] {$\displaystyle g( I_{2})$};

\end{tikzpicture}

\end{center}
\end{exmp}

Notice in the previous example that the $g$-vectors of all indecomposable modules and shifted projectives lie on a wall in the stability picture. This is indeed not a coincidence and was proven in more generality by Br\"{u}stle, Smith, and Treffinger in [\ref{ref: BT wall and chamber structure}].

\begin{thm}
The $g$-vectors of indecomposable rigid modules and shifted projectives lie on walls in the stability picture.
\end{thm}

\begin{exmp}\label{exmp: stereographic projection}
Take $Q = 1\leftarrow 2\leftarrow 3$. Then a stereographic projection of the stability picture is depicted below. The $i$th coordinate is negative inside the wall of the $i$th simple, it is positive outside the wall, and it is zero on the wall. The vertices depicted are the $g$-vectors of the corresponding indecomposable modules and their shifted projectives. Notice that there are 14 chambers including the trivial outer chamber, each bounded by three $g$-vectors of three indecomposable modules or shifted projectives. This is no coincidence since the objects corresponding to the vertices of each chamber form a silting pair and hence a cluster. Notice that the stability picture is actually homeomorphic to the extended compatibility graph that we have already computed by straightening the walls that connect any two $g$-vectors.

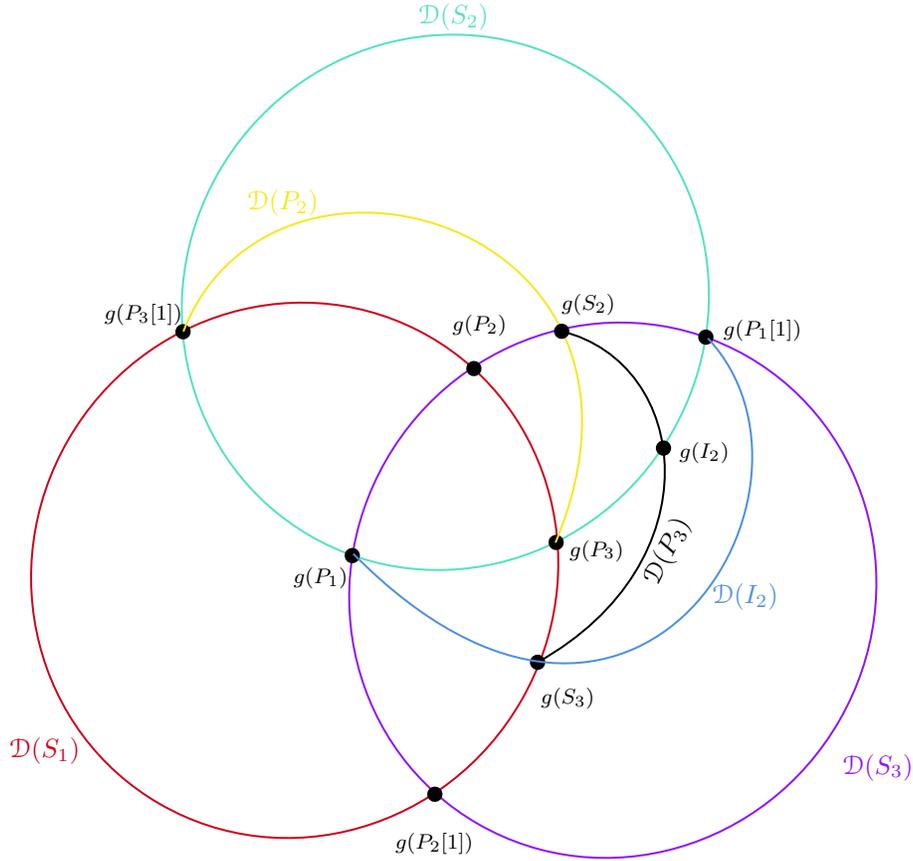
\begin{center}

\tikzset{every picture/.style={line width=0.75pt}} 

\begin{tikzpicture}[x=0.75pt,y=0.75pt,yscale=-1,xscale=1]

\draw  [color={rgb, 255:red, 208; green, 2; blue, 27 }  ,draw opacity=1 ] (170.58,444.02) .. controls (107.45,407.73) and (86.65,325.47) .. (124.12,260.28) .. controls (161.59,195.09) and (243.14,171.66) .. (306.27,207.94) .. controls (369.4,244.23) and (390.2,326.49) .. (352.73,391.68) .. controls (315.26,456.87) and (233.71,480.3) .. (170.58,444.02) -- cycle ;
\draw  [color={rgb, 255:red, 80; green, 227; blue, 194 }  ,draw opacity=1 ] (246.56,308.75) .. controls (183.43,272.47) and (162.63,190.2) .. (200.1,125.01) .. controls (237.57,59.83) and (319.12,36.39) .. (382.25,72.68) .. controls (445.38,108.96) and (466.18,191.23) .. (428.71,256.41) .. controls (391.24,321.6) and (309.69,345.04) .. (246.56,308.75) -- cycle ;
\draw  [color={rgb, 255:red, 144; green, 19; blue, 254 }  ,draw opacity=1 ] (331.01,454.06) .. controls (267.88,417.78) and (247.08,335.52) .. (284.55,270.33) .. controls (322.02,205.14) and (403.57,181.71) .. (466.7,217.99) .. controls (529.83,254.28) and (550.63,336.54) .. (513.16,401.73) .. controls (475.69,466.92) and (394.14,490.35) .. (331.01,454.06) -- cycle ;
\draw    (182.11,205.56) ;
\draw [shift={(182.11,205.56)}, rotate = 0] [color={rgb, 255:red, 0; green, 0; blue, 0 }  ][fill={rgb, 255:red, 0; green, 0; blue, 0 }  ][line width=0.75]      (0, 0) circle [x radius= 3.35, y radius= 3.35]   ;
\draw    (328.79,224.14) -- (329.58,222.77) ;
\draw [shift={(328.79,224.14)}, rotate = 299.89] [color={rgb, 255:red, 0; green, 0; blue, 0 }  ][fill={rgb, 255:red, 0; green, 0; blue, 0 }  ][line width=0.75]      (0, 0) circle [x radius= 3.35, y radius= 3.35]   ;
\draw    (267.47,318.53) -- (266.14,317.76) ;
\draw [shift={(267.47,318.53)}, rotate = 209.89] [color={rgb, 255:red, 0; green, 0; blue, 0 }  ][fill={rgb, 255:red, 0; green, 0; blue, 0 }  ][line width=0.75]      (0, 0) circle [x radius= 3.35, y radius= 3.35]   ;
\draw    (370.29,311.9) -- (368.96,311.14) ;
\draw [shift={(370.29,311.9)}, rotate = 209.89] [color={rgb, 255:red, 0; green, 0; blue, 0 }  ][fill={rgb, 255:red, 0; green, 0; blue, 0 }  ][line width=0.75]      (0, 0) circle [x radius= 3.35, y radius= 3.35]   ;
\draw [color={rgb, 255:red, 248; green, 231; blue, 28 }  ,draw opacity=1 ]   (182.29,205.43) .. controls (226.76,90.41) and (437.99,148.9) .. (369.96,311.97) ;
\draw    (309.12,439) -- (309.7,437.99) ;
\draw [shift={(309.12,439)}, rotate = 299.89] [color={rgb, 255:red, 0; green, 0; blue, 0 }  ][fill={rgb, 255:red, 0; green, 0; blue, 0 }  ][line width=0.75]      (0, 0) circle [x radius= 3.35, y radius= 3.35]   ;
\draw    (445.79,208.28) ;
\draw [shift={(445.79,208.28)}, rotate = 0] [color={rgb, 255:red, 0; green, 0; blue, 0 }  ][fill={rgb, 255:red, 0; green, 0; blue, 0 }  ][line width=0.75]      (0, 0) circle [x radius= 3.35, y radius= 3.35]   ;
\draw    (424.43,264.28) ;
\draw [shift={(424.43,264.28)}, rotate = 0] [color={rgb, 255:red, 0; green, 0; blue, 0 }  ][fill={rgb, 255:red, 0; green, 0; blue, 0 }  ][line width=0.75]      (0, 0) circle [x radius= 3.35, y radius= 3.35]   ;
\draw    (373.09,205.32) .. controls (438.27,221.36) and (450.65,322.2) .. (360.94,372.38) ;
\draw    (373.09,205.32) ;
\draw [shift={(373.09,205.32)}, rotate = 0] [color={rgb, 255:red, 0; green, 0; blue, 0 }  ][fill={rgb, 255:red, 0; green, 0; blue, 0 }  ][line width=0.75]      (0, 0) circle [x radius= 3.35, y radius= 3.35]   ;
\draw    (360.94,372.38) ;
\draw [shift={(360.94,372.38)}, rotate = 0] [color={rgb, 255:red, 0; green, 0; blue, 0 }  ][fill={rgb, 255:red, 0; green, 0; blue, 0 }  ][line width=0.75]      (0, 0) circle [x radius= 3.35, y radius= 3.35]   ;
\draw [color={rgb, 255:red, 74; green, 144; blue, 226 }  ,draw opacity=1 ]   (267.97,317.61) .. controls (411.6,463.09) and (517.18,284.12) .. (446.81,208.86) ;

\draw (93.06,408.93) node [anchor=north west][inner sep=0.75pt]  [font=\footnotesize,color={rgb, 255:red, 208; green, 2; blue, 27 }  ,opacity=1 ]  {$\mathcal{D}( S_{1})$};
\draw (299.21,38.34) node [anchor=north west][inner sep=0.75pt]  [font=\footnotesize,color={rgb, 255:red, 80; green, 227; blue, 194 }  ,opacity=1 ]  {$\mathcal{D}( S_{2})$};
\draw (513.3,417.08) node [anchor=north west][inner sep=0.75pt]  [font=\footnotesize,color={rgb, 255:red, 144; green, 19; blue, 254 }  ,opacity=1 ]  {$\mathcal{D}( S_{3})$};
\draw (212.92,132.06) node [anchor=north west][inner sep=0.75pt]  [font=\footnotesize,color={rgb, 255:red, 248; green, 231; blue, 28 }  ,opacity=1 ]  {$\mathcal{D}( P_{2})$};
\draw (410.48,329.96) node [anchor=north west][inner sep=0.75pt]  [font=\footnotesize,rotate=-296.74]  {$\mathcal{D}( P_{3})$};
\draw (447.44,330.99) node [anchor=north west][inner sep=0.75pt]  [font=\footnotesize,color={rgb, 255:red, 74; green, 144; blue, 226 }  ,opacity=1 ]  {$\mathcal{D}( I_{2})$};
\draw (375.61,309.17) node [anchor=north west][inner sep=0.75pt]  [font=\scriptsize]  {$g( P_{3})$};
\draw (430.97,259.66) node [anchor=north west][inner sep=0.75pt]  [font=\scriptsize]  {$g( I_{2})$};
\draw (140.95,190.15) node [anchor=north west][inner sep=0.75pt]  [font=\scriptsize]  {$g( P_{3}[ 1])$};
\draw (361.2,383.76) node [anchor=north west][inner sep=0.75pt]  [font=\scriptsize]  {$g( S_{3})$};
\draw (316.56,195.69) node [anchor=north west][inner sep=0.75pt]  [font=\scriptsize]  {$g( P_{2})$};
\draw (236.34,322.9) node [anchor=north west][inner sep=0.75pt]  [font=\scriptsize]  {$g( P_{1})$};
\draw (287.74,456.95) node [anchor=north west][inner sep=0.75pt]  [font=\scriptsize]  {$g( P_{2}[ 1])$};
\draw (453.23,197.88) node [anchor=north west][inner sep=0.75pt]  [font=\scriptsize]  {$g( P_{1}[ 1])$};
\draw (371.54,185.06) node [anchor=north west][inner sep=0.75pt]  [font=\scriptsize]  {$g( S_{2})$};

\end{tikzpicture}

\end{center}

\end{exmp}

In general, there is a bijection between chambers in the stability picture and silting pairs proven by Br\"{u}stle, Smith, and Treffinger in [\ref{ref: BT wall and chamber structure}].

\begin{thm}
The chambers in the stability picture are in bijection with silting pairs, and hence clusters, where each chamber is sent to the direct sum of the modules/shifted projectives whose $g$-vectors enclose said chamber.
\end{thm}

It follows from this theorem that two $g$-vectors that lie on a wall in the stability picture with no $g$-vector between them form a rigid pair. Thus the extended compatibility graph and the stability picture are homeomorphic for any finite dimensional hereditary algebra. Notice that the wall and chamber structure in the previous example triangulates the 2-sphere. This was proven true by Demonet, Iyama, and Jasso in [\ref{ref: DIJ triangulation of sphere}] for all algebras that have finitely many silting pairs. 

\begin{thm}
The compatibility graph and stability picture of a hereditary algebra with finitely many silting pairs form a triangulation of an $n-1$ sphere.
\end{thm}

\section{Rigidity}\label{sec: rigidity}

\indent 

One may question whether the stability picture in Example \ref{exmp: stereographic projection} is accurate; that is, whether there is nothing inside $\mathcal{D}(S_1)$ and outside both $\mathcal{D}(S_3)$ and $\mathcal{D}(S_2)$ for instance. The answer to this question is yes, the diagram is accurate and the reason for this is because `rigidity is an open condition'. We will spend this section proving this fact, but first, we begin with an example.  

\begin{exmp}
Take $\Bbbk = \mathbb{C}$ and $Q: 1\rightarrow 2$. Then the collection of all representations of $Q$ with dimension $(3,2)$ is $$\text{Rep}_{\mathbb{C}Q}(3,2) := \{ \mathbb{C}^3 \overset{f}{\rightarrow} \mathbb{C}^2 \} \cong \mathbb{C}^6.$$ The morphism $f$ is given by a $3\times 2$ matrix with entries in $\mathbb{C}$, say $\begin{bmatrix} a & c & e \\ b & d & f \end{bmatrix}$. The condition of $f$ being onto is equivalent to this matrix having full rank of 2, which is further equivalent to the non-vanishing of the minors: $ad-bc \neq 0$, $cf-de\neq 0$, and $af - be \neq 0$. Since these are polynomial inequalities, $\{f : f \, \text{is onto} \}$ is a Zariski open subset of $\text{Rep}_{\mathbb{C}Q}(3,2) \cong \mathbb{C}^6$. With the usual topology on $\mathbb{C}^6$, we see that $\{f : f \text{ is onto} \}$ is open, dense, and has full measure. \\

If $f$ is onto, then the corresponding representation decomposes:  $$\mathbb{C}^3 \twoheadrightarrow \mathbb{C}^2 \cong \mathbb{C}^2 \overset{\cong}{\rightarrow} \mathbb{C}^2 \oplus \mathbb{C} \rightarrow 0 = 2P_1 \oplus S_1.$$
Therefore all the $f$ in this open, dense, full measure set give the same rigid module, namely $2P_1 \oplus S_1$.
\end{exmp}

It is this idea that surjectivity of a morphism is a Zariski open condition, illuminated by the previous example, that will play a key role in showing that the stability picture in Example \ref{exmp: stereographic projection} is accurate. To continue, we must further develop the notion of 2-term silting complexes. 

\subsection{Category of 2-term Silting Complexes}

\indent

The notion of a silting object in a category was first used studied by Keller and Vossieck in [\ref{ref: KV intro silting}]. The notion of two-term silting complexes was first studied several years later by Hoshino, Kato, and Miyachi in [\ref{ref: HKM Two term silting complexes}]. Some time later, Adachi, Iyama, and Reiten in [\ref{ref: AIR tau-tilting}] noticed that the zero cohomologies of such complexes connect with cluster theory and more generally, $\tau$-tilting theory. For more on the topic, we suggest H\"{u}gel's survey [\ref{ref: H silting objects}]. \\

The \textbf{category of 2-term silting complexes} is a subcategory of the bounded derived category of $\Bbbk Q$. Its objects, called \textbf{2-term silting complexes}, are segments of chain complexes of the form $C_1\overset{P}{\rightarrow}C_0$ where $C_0$ and $C_1$ are both projective $\Bbbk Q$-modules. The morphisms are morphisms of chain complexes up to chain homotopy. More concretely, they are pairs of module morphisms $f_* = (f_1,f_0)$ up to equivalence such that the following diagram commutes:

\begin{center}
\begin{tabular}{c}
\xymatrix{ C_1 \ar_{p}[d] \ar^{f_1}[r] & D_1 \ar^{q}[d] \\  C_0 \ar_{f_0}[r] & D_0 }
\end{tabular}
\end{center}

Two morphisms $f_*$ and $g_*$ are equivalent, or chain homotopic denoted by $f_* \simeq g_*$, if there exists a module morphism $h:C_0 \rightarrow D_1$ such that $qh = g_0 - f_0$ and $hp = g_1 - f_1$. The following diagram may be useful.

\begin{center}
\begin{tabular}{c}
\xymatrix{ C_1 \ar_{p}[dd] \ar@<-.5ex>_{g_1}[rr] \ar@<.5ex>^{f_1}[rr] & & D_1 \ar^{q}[dd] \\ & & \\  C_0 \ar^h[uurr] \ar@<-.5ex>_{g_0}[rr] \ar@<.5ex>^{f_0}[rr] & & D_0 }
\end{tabular}
\end{center}
The morphism $h$ is a \textbf{chain homotopy}. The indecomposable objects of this category are completely described:

\begin{thm}
The indecomposable objects in the category of 2-term silting complexes are 
\begin{enumerate}
\item Projective presentations of indecomposable $\Bbbk Q$-modules $M$: $C_1 \overset{p}{\hookrightarrow} C_0$ where the cokernel of $p$ is $M$. 
\item Objects of the form $P \rightarrow 0$ where $P$ is an indecomposable projective, that is, the projective presentations of the shifted projectives.
\end{enumerate}
\end{thm}

\begin{proof}[Sketch of proof]
Let $C_1 \overset{p}{\rightarrow} C_0$ be a 2-term silting complex. Then we have a commutative diagram:

\begin{center}
\begin{tabular}{c}
\xymatrix{ C_1 \ar^{p}[rr] \ar@{->>}[dr] & & C_0  \\  & \text{im}(p) \ar@{^{(}->}[ur]
 & }
\end{tabular}
\end{center}
Since $\Bbbk Q$ is hereditary, we have that im$(p)$ is a projective module, so we have a projective presentation of $M = \text{coker}(p)$ given by $0 \rightarrow \text{im}(p) \rightarrow C_0 \rightarrow M \rightarrow 0$. Moreover, the exact sequence $0 \rightarrow \text{ker}(p) \rightarrow C_1 \rightarrow \text{im}(p) \rightarrow 0$ splits by the projectivity of im$(p)$. Therefore, we can decompose the original 2-term silting complex into the direct sum of projective presentations of a shifted projective and $M$ respectively:

$$ C_1 \overset{p}{\rightarrow} C_0 \cong \text{ker}(p) \rightarrow 0 \oplus \text{im}(p) \rightarrow C_0.$$ \end{proof}

We will now recall some facts from homological algebra and category theory.

\begin{lem}
The projective presentation of a $\Bbbk Q$-module $M$ is unique up to chain homotopy.
\end{lem}

\begin{thm}
Let $C_1 \rightarrow C_0$ and $D_1 \rightarrow D_0$ be two projective presentations of $M$ and $N$ respectively. Then Ext$^1(M,N)$ is the set of homotopy classes of maps $h:C_1 \rightarrow D_0$.
\end{thm}

What this theorem is saying is that Ext$(M,N)$ is equivalent to the set of all $h$ up to homotopy of the form 

\begin{center}
\begin{tabular}{c}
\xymatrix{ 0 \ar[dd] \ar[rr] & & D_1 \ar^q[dd] \\ & & \\ C_1\ar_{p}[dd] \ar^{g}[uurr] \ar^{h}[rr] & & D_0 \ar[dd] \\ & & \\ C_0 \ar^f[uurr] \ar[rr] & & 0 }
\end{tabular}
\end{center}

In this diagram, we have shifted the projective presentation of $N$ by one. Really what we are looking at here is Hom$(M,N[1])$ in the derived category, which is well known to be Ext$(M,N)$; however, from this diagram we can conclude that Ext$(M,N)$ is the cokernel of the map Hom$_{\Bbbk Q}(C_0,D_0) \oplus$ Hom$_{\Bbbk Q}(C_1,D_1) \overset{(p,q)}{\rightarrow} \text{Hom}_{\Bbbk Q}(C_1,D_0)$ that sends $(f,g) \mapsto f\circ p + q \circ g$. If Ext$(M,N) = 0$, then this map is surjective, a Zariski open condition on the maps $p$ and $q$. We have the following corollary whose proof is written down by Igusa, Orr, Todorov, and Weyman in [\ref{ref: IOTW rigidity is open}].

\begin{cor}\label{cor: open condition}
Suppose Ext$(M,N) = 0$. Then there exist open neighborhoods $U$ and $V$ of $p$ and $q$ respectively in Hom$_{\Bbbk Q}(C_1,C_0)$ and Hom$_{\Bbbk Q}(D_1,D_0)$ such that all $f \in U$ and $g \in V$ are monomorphisms whose cokernels satisfy Ext$(\text{coker}f, \text{coker}g) = 0$.
\end{cor}

To continue, we require the following lemma originally due to Happel and Ringel in [\ref{ref: HR happel-ringel lemma}]. A proof of which can also be found in [\ref{ref: Bill's notes}].

\begin{lem}[Happel-Ringel]
Let $A, B$ be two indecomposable $\Bbbk Q$-modules such that Ext$(B,A) = 0$. Then any nonzero morphism $f:A\rightarrow B$ is either mono or epi.
\end{lem}

\begin{proof}
Let $f:A\rightarrow B$ be a morphism and let $C = \text{im}f$ and $X = \text{ker}f$. Then we have a short exact sequence $0\rightarrow X \rightarrow A \overset{g}{\rightarrow}C \rightarrow 0$. This induces a long exact sequence that contains the chain complex Ext$(B/C,X) \rightarrow \text{Ext}(B/C,A) \overset{g_\#}{\rightarrow} \text{Ext}(B/C,C) \rightarrow 0$ where the last term vanishes because $\Bbbk Q$ is hereditary. We conclude that the map $g_{\#}$ is a surjection. In particular, this implies that there exists a module $D$ such that we have the following diagram:

\begin{center}
\begin{tabular}{c}
\xymatrix{ 0 \ar[r]  & A \ar_{g}[d] \ar[r] & D \ar[d] \ar[r] & B/C \ar_1[d] \ar[r] & 0 \\0 \ar[r] & C \ar[r] & B \ar[r] & B/C \ar[r] & 0}
\end{tabular}
\end{center}
This yields a Mayer-Vietoris sequence in which every third map is an isomorphism, hence we attain a short exact sequence $0 \rightarrow A \rightarrow C \oplus D \rightarrow B \rightarrow 0 \in \text{Ext}(B,A)$. Since $\text{Ext}(B,A) = 0$ by assumption, we have that $C \oplus D \cong A \oplus B$. We conclude that $C = A$ in which case $f$ is mono, or $C = B$, in which case $f$ is epi. \end{proof} 

An example of how we can use this lemma is the following.

\begin{exmp}
Let $Q = 1\leftarrow 2\leftarrow 3$. Then there is a map from $P_2$ to $I_2$, namely $(0,1,0)$. This map is neither mono nor epi, so by the Happel-Ringel lemma we conclude that there is a nontrivial extension of $I_2$ by $P_1$, namely $P_2 \hookrightarrow S_2 \oplus P_3 \twoheadrightarrow I_2$, which is an almost split sequence.
\end{exmp}

Moreover, if we consider a rigid module $M$, then any map from $M$ to itself is either mono or epi by the Happel-Ringel lemma. Then we have the following corollary which can also be found in [\ref{ref: Bill's notes}].

\begin{cor}\label{cor: rigid and indec implies brick}
If $M$ is rigid and indecomposable, then End$(M) = \Bbbk$; that is, $M$ is a brick.
\end{cor}

\subsection{Rigidity is Open}

\indent 

The next and main result of this section provides some intuition behind why the term rigid is chosen. Intuitively it comes down to the fact that if you shake $M$ a little bit, $M$ won't change. The proof is written by Igusa, Orr, Todorov, and Weyman in [\ref{ref: IOTW rigidity is open}].

\begin{thm}\label{thm: rigidity is open}
Let $0 \rightarrow P' \rightarrow P \rightarrow M \rightarrow 0$ be a projective presentation of a rigid $\Bbbk Q$ module $M$. Then there is a Zariski open neighborhood $U \subset \text{Hom}(P',P)$ such that all $f \in U$, are mono and $M_f := \text{coker}f \cong M$. In other words, If $M$ is rigid, almost all modules $M'$ with the same dimension vector as $M$ are isomorphic to $M$.
\end{thm}

\begin{proof}[Sketch of Proof]
Suppose $M$ is indecomposable, we will prove the general case in Section \ref{sec: maximal green sequences}. By Corollary \ref{cor: open condition}, we know there is an open $U \subset \text{Hom}(P',P)$ such that all $f,g \in U$ are mono and Ext$(M_f,M_g) = 0$. It remains to show that $M_f$ and $M_g$ are isomorphic. We have two short exact sequences $$0 \rightarrow P' \rightarrow P \rightarrow M_g \rightarrow 0$$ $$0 \rightarrow P' \rightarrow P \rightarrow M_f \rightarrow 0$$

These short exact sequences allow us to conclude that the dimension vectors of $M_f$ and $M_g$ are the same. Moreover, they induce long exact sequences $$0 \rightarrow \text{Hom}(M_g,M_f) \rightarrow \text{Hom}(P,M_f) \overset{g^\#}{\rightarrow} \text{Hom}(P',M_f) \rightarrow 0$$ $$0 \rightarrow \text{Hom}(M_f,M_f) \rightarrow \text{Hom}(P,M_f) \overset{f^\#}{\rightarrow} \text{Hom}(P',M_f) \rightarrow 0$$

\noindent
where the last term is zero since there are no extensions. Since $M_f$ is rigid and indecomposable, it is a brick by Corollary \ref{cor: rigid and indec implies brick}, so $\text{Hom}(M_f,M_f) = \Bbbk$. This forces Hom$(M_g,M_f) \neq 0$, so there is a nontrivial morphism $h:M_g \rightarrow M_f$ that is either mono or epi by the Happel-Ringel lemma. But since these two modules have the same dimension vector, it must be an isomorphism. 
\end{proof}

We will now list some consequences of this theorem, all of which have also been proven using different methods by Br\"{u}stle, and Treffinger in [\ref{ref: BT wall and chamber structure}] and by Adachi, Iyama, and Reiten in [\ref{ref: AIR tau-tilting}].

\begin{cor}
The components $T_i$ of a rigid module have linearly independent $g$-vectors and dimension vectors.
\end{cor}

\begin{proof}[Sketch of Proof]
Since dim$M = C_{\Bbbk Q} g(M)$ where $C_{\Bbbk Q}$ is the Cartan matrix, it suffices to prove this for the $g$-vectors. For a contradiction, suppose that the components of a rigid module have linearly dependent dimension vectors. Then $\sum n_i g(T_i) = 0$. We then collect the positive and negative terms as in the following example. Suppose for example we have $2g(T_1) + g(T_2) = 3g(T_3) + g(T_4) = (4,-2,-1,5)$. Then this implies we have two maps $\xymatrix{2P_2 \oplus P_3 \ar@<.5ex>^f[r] \ar@<-.5ex>_g[r] & 4P_1 \oplus 5 P_4}$ such that coker$f = 2T_1 \oplus T_2$ and coker$g = 3T_3 \oplus T_4$. By the previous theorem, we have Zariski open $U,V \subset \text{Hom}(2P_2 \oplus P_3,4P_1 \oplus 5 P_4)$ such that all $h \in U$ have the same cokernel as $f$ and all $j \in V$ have the same cokernel as $g$. Since $U$ and $V$ are Zariski open, their intersection is not empty. But then the previous theorem implies that $3T_3 \oplus T_4 \cong 2T_1 \oplus T_2$, a contradiction.
\end{proof}

The following three facts are immediate consequences from the previous corollary.

\begin{cor} {\color{white} .}
\begin{enumerate}
\item Rigid modules can have at most $n$ nonisomorphic indecomposable summands.
\item The $g$-vectors of the components of a silting pair span an $n-1$ simplex. 
\item Simplices can't overlap.
\end{enumerate}
\end{cor}

In particular, the fact that these simplices don't overlap allows us to conclude that the stability pictures we have drawn in Section \ref{sec: stability pictures} are indeed accurate. Moreover, it implies that we cannot have simplices of the following form in our extended compatibility graph.

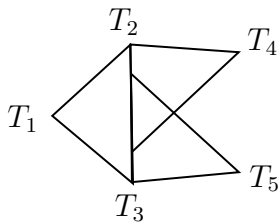
\begin{center}

\tikzset{every picture/.style={line width=0.75pt}} 

\begin{tikzpicture}[x=0.75pt,y=0.75pt,yscale=-1,xscale=1]

\draw   (263,105.11) -- (302.27,69.29) -- (303.72,139.28) -- cycle ;
\draw   (302.45,69.1) -- (357.02,72.98) -- (303.18,123.31) -- cycle ;
\draw   (303.16,138.55) -- (302.72,83.84) -- (357.14,133.54) -- cycle ;

\draw (239,98.4) node [anchor=north west][inner sep=0.75pt]    {$T_{1}$};
\draw (290,49.4) node [anchor=north west][inner sep=0.75pt]    {$T_{2}$};
\draw (293,143.4) node [anchor=north west][inner sep=0.75pt]    {$T_{3}$};
\draw (360,59.4) node [anchor=north west][inner sep=0.75pt]    {$T_{4}$};
\draw (361,128.4) node [anchor=north west][inner sep=0.75pt]    {$T_{5}$};

\end{tikzpicture}

\end{center}

If we did have such simplices, then the mutation of $T_1$ would not be well defined since we would not know whether to send it to $T_4$ or $T_5$.

\subsection{Stable Barcode}

\indent 

In this subsection, we will present a connection between rigidity and persistent homology through stable barcodes. For more on the topic we suggest Vejdemo-Johansson's survey on the topic [\ref{ref: V-J survey on persistent homology}]. Throughout this subsection, let $Q$ be the linearly oriented quiver of type $\mathbb{A}_n$ given by $1 \leftarrow 2 \leftarrow \dots \leftarrow n$. Recall that the indecomposable modules are uniquely determined by their dimension vector. This leads to the notion of \textbf{interval modules} where $M_{ab}$ denotes the indecomposable module with support $[a,b]$. We will construct the stable barcode associated to the rigid modules. To do this, we need a lemma that classifies extensions of interval modules.

\begin{lem}
Ext$(M_{cd},M_{ab}) \neq 0$ if and only if one of the following hold:
\begin{enumerate}
\item $a<c\leq b < d$.
\item $b+1 = c$.
\end{enumerate}
\end{lem}

In case 1 we have a short exact sequence given by $M_{ab} \hookrightarrow M_{ad} \oplus M_{cb} \twoheadrightarrow M_{cd}$. In case 2 we have an exact sequence $M_{ab} \hookrightarrow M_{ad}  \twoheadrightarrow M_{cd}$. This lemma allows us to prove the following theorem.

\begin{thm}
The unique rigid module with dimension vector $v$ is given by placing $v_i$ spots above the point $(i,0)$ on the $x$-axis and joining adjacent spots horizontally. This is called the \textbf{stable barcode} associated to $v$.
\end{thm}

\begin{exmp}
Take $v = (3,4,2)$. Then the corresponding rigid module is $M_{22} \oplus M_{12} \oplus 2M_{13}$ and the stable barcode is as follows.
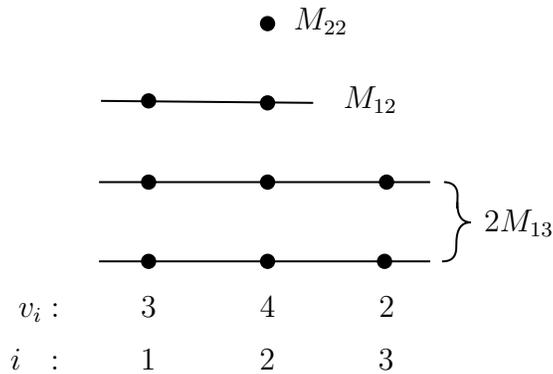
\begin{center}

\tikzset{every picture/.style={line width=0.75pt}} 

\begin{tikzpicture}[x=0.75pt,y=0.75pt,yscale=-1,xscale=1]

\draw    (237.22,198) -- (238.22,198) ;
\draw [shift={(238.22,198)}, rotate = 0] [color={rgb, 255:red, 0; green, 0; blue, 0 }  ][fill={rgb, 255:red, 0; green, 0; blue, 0 }  ][line width=0.75]      (0, 0) circle [x radius= 3.35, y radius= 3.35]   ;
\draw    (298.22,198) ;
\draw [shift={(298.22,198)}, rotate = 0] [color={rgb, 255:red, 0; green, 0; blue, 0 }  ][fill={rgb, 255:red, 0; green, 0; blue, 0 }  ][line width=0.75]      (0, 0) circle [x radius= 3.35, y radius= 3.35]   ;
\draw    (357.22,198) ;
\draw [shift={(357.22,198)}, rotate = 0] [color={rgb, 255:red, 0; green, 0; blue, 0 }  ][fill={rgb, 255:red, 0; green, 0; blue, 0 }  ][line width=0.75]      (0, 0) circle [x radius= 3.35, y radius= 3.35]   ;
\draw    (238.22,158) ;
\draw [shift={(238.22,158)}, rotate = 0] [color={rgb, 255:red, 0; green, 0; blue, 0 }  ][fill={rgb, 255:red, 0; green, 0; blue, 0 }  ][line width=0.75]      (0, 0) circle [x radius= 3.35, y radius= 3.35]   ;
\draw    (298.22,158) ;
\draw [shift={(298.22,158)}, rotate = 0] [color={rgb, 255:red, 0; green, 0; blue, 0 }  ][fill={rgb, 255:red, 0; green, 0; blue, 0 }  ][line width=0.75]      (0, 0) circle [x radius= 3.35, y radius= 3.35]   ;
\draw    (358.22,157) -- (358.22,158) ;
\draw [shift={(358.22,158)}, rotate = 90] [color={rgb, 255:red, 0; green, 0; blue, 0 }  ][fill={rgb, 255:red, 0; green, 0; blue, 0 }  ][line width=0.75]      (0, 0) circle [x radius= 3.35, y radius= 3.35]   ;
\draw    (238.22,117) ;
\draw [shift={(238.22,117)}, rotate = 0] [color={rgb, 255:red, 0; green, 0; blue, 0 }  ][fill={rgb, 255:red, 0; green, 0; blue, 0 }  ][line width=0.75]      (0, 0) circle [x radius= 3.35, y radius= 3.35]   ;
\draw    (298.22,118) ;
\draw [shift={(298.22,118)}, rotate = 0] [color={rgb, 255:red, 0; green, 0; blue, 0 }  ][fill={rgb, 255:red, 0; green, 0; blue, 0 }  ][line width=0.75]      (0, 0) circle [x radius= 3.35, y radius= 3.35]   ;
\draw    (298.22,79) -- (298.22,78) ;
\draw [shift={(298.22,78)}, rotate = 270] [color={rgb, 255:red, 0; green, 0; blue, 0 }  ][fill={rgb, 255:red, 0; green, 0; blue, 0 }  ][line width=0.75]      (0, 0) circle [x radius= 3.35, y radius= 3.35]   ;
\draw    (213.22,198) -- (380.22,198) ;
\draw    (213.22,158) -- (380.22,158) ;
\draw    (214.22,117) -- (321.22,118) ;
\draw   (386.22,198) .. controls (390.89,198.12) and (393.28,195.85) .. (393.39,191.18) -- (393.47,188.18) .. controls (393.64,181.51) and (396.05,178.24) .. (400.72,178.35) .. controls (396.05,178.24) and (393.8,174.85) .. (393.97,168.18)(393.89,171.18) -- (394.04,165.18) .. controls (394.16,160.51) and (391.89,158.12) .. (387.22,158) ;

\draw (335,108.4) node [anchor=north west][inner sep=0.75pt]    {$M_{12}$};
\draw (310,68.4) node [anchor=north west][inner sep=0.75pt]    {$M_{22}$};
\draw (406,170.4) node [anchor=north west][inner sep=0.75pt]    {$2M_{13}$};
\draw (171,214.4) node [anchor=north west][inner sep=0.75pt]    {$v_{i} :\ \ \ \ \ \ \ 3\ \ \ \ \ \ \ \ \ \ 4\ \ \ \ \ \ \ \ \ \ 2$};
\draw (167,240.4) node [anchor=north west][inner sep=0.75pt]    {$i\ \ :\ \ \ \ \ \ \ 1\ \ \ \ \ \ \ \ \ \ 2\ \ \ \ \ \ \ \ \ \ 3$};

\end{tikzpicture}

\end{center}
\end{exmp}

\section{Maximal Green Sequences}\label{sec: maximal green sequences}

\indent

In this section we will begin by removing the indecomposability assumption in the proof of Theorem \ref{thm: rigidity is open}. We begin with a lemma that was first proven by Schofield in [\ref{ref: S general reps}], then proved in the language used here by Igusa and Schiffler in [\ref{ref: IS Clusters form exceptional collections}]. 

\begin{lem}\label{lem: Schofield}
Let $T_1, T_2, \dots , T_k$ be indecomposable ext-orthogonal $\Bbbk Q$-modules; that is, Ext$(T_i,T_j) = 0$ for all $i$ and $j$. Then the $T_i$ can be ordered such that Hom$(T_j,T_i) = 0$ for all $i<j$. Such a sequence of modules is an example of an \textbf{exceptional sequence}, which is a sequence of modules $(M_1, M_2, \dots , M_k)$ such that Hom$(M_j,M_i) = 0 =$ Ext$(M_j,M_i)$ for all $i<j$.
\end{lem}

\begin{proof}[Sketch of Proof]
It suffices to show that there are no oriented cycles of hom: 
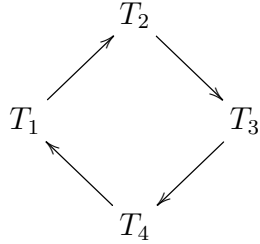
\begin{center}
\begin{tabular}{c}
\xymatrix{ & T_2 \ar[dr] & \\T_1 \ar[ur] & & T_3 \ar[dl] \\ & T_4 \ar[ul] & }
\end{tabular}
\end{center}
If there were such a cycle, by the Happel-Ringel lemma, each map is a monomorphism or an epimorphism. This forces an epimorphism followed by a monomorphism whose composition is nonzero and neither mono nor epi, a contradiction of the Happel-Ringel lemma.
\end{proof}

\begin{cor}
The set $\{T_1, T_2, \dots , T_k\}$ form an exceptional collection, that is, there exists at least one ordering $(T_{i_1}, T_{i_2}, \dots , T_{i_k})$ that is an exceptional sequence. 
\end{cor}

We are now ready to prove Theorem \ref{thm: rigidity is open}.

\begin{proof}[Proof Sketch of Theorem \ref{thm: rigidity is open}]
Let $M$ be a rigid $\Bbbk Q$-module and $P_1 \overset{p}{\rightarrow} P_0 \rightarrow M$ be a projective presentation of $M$. By Corollary \ref{cor: open condition}, there exists an open $U\subset \text{Hom}(P_1,P_0)$ such that any $f,g \in U$ is such that Ext$(M_f,M_g) = 0$ where $M_f$ is the cokernel of $f$ and similarly for $g$. Therefore, $M_f \oplus M_g$ is rigid. Let $E_1, E_2, \dots E_k$ be the list of nonisomorphic indecomposable summands of $M_f \oplus M_g$. By Lemma \ref{lem: Schofield}, we may without loss of generality assume that Hom$(E_j,E_i) = 0$ for $j > i$. By applying Hom$(-,E_i)$ to the short exact sequence $0 \rightarrow P_1 \overset{f}{\rightarrow} P_0 \rightarrow M_f \rightarrow 0$ and the analogous one for $g$, we get the long exact sequences 

 $$ 0\rightarrow \text{Hom}(M_f,E_i) \rightarrow \text{Hom}(P_0,E_i) \rightarrow \text{Hom}(P_1,E_i) \rightarrow 0$$
 $$ 0\rightarrow \text{Hom}(M_g,E_i) \rightarrow \text{Hom}(P_0,E_i) \rightarrow \text{Hom}(P_1,E_i) \rightarrow 0$$

\noindent
where the last term is zero since Ext$(M_f,E_i) = 0$ and analogously for $g$. By analyzing dimensions we conclude that dim(Hom$(M_f,E_i)) =$ dim(Hom$(M_g,E_i)) = \Bbbk^{a_i}$. Suppose that $M_f = b_1E_1 \oplus b_2E_2 \oplus \dots \oplus b_kE_k$. Then we have that $a_1 = b_1, a_2 = b_2 + c_{12}b_1, $ and so on where $c_{ij} = \text{dim(Hom(}E_i,E_j))$. Therefore the $a_i$ determine the $b_i$ and we conclude $M_f \cong M_g$.
\end{proof}

\subsection{Sign Coherence of $g$-vectors}

\indent 

The sign coherence of $g$-vectors was a conjecture until proven in [\ref{ref: GHKK sign coherence c-vectors}] by Gross, Hacking, Keel, and Kontsevich for skew-symmetrizable cluster algebras. The now theorem is as follows.

\begin{thm}
For any tilting module $T = T_1 \oplus T_2 \oplus \dots \oplus T_n$ and any $k$, the $k$th coordinate of $g(T_i)$ have the same sign.
\end{thm}

\begin{exmp}
Consider $Q: 1 \leftarrow 2 \leftarrow 3$ and $T = I_2 \oplus S_2 \oplus P_3$. Then by placing the corresponding $g$-vectors in the first, second, and third column respectively of a matrix, we get the following.
\begin{center}
$\begin{bmatrix}
-1 & -1 & 0 \\ 0 & 1 & 0 \\ 1 & 0 & 1
\end{bmatrix}$
\end{center}

Note in each row, the signs of all the entries are the same (or zero). The first row is all negative and the last two are all positive.

\end{exmp}

The sign coherence of $g$-vectors is equivalent to the statement that $D(S_k) = \{x\in\mathbb{R}^n : x_k = 0\}$ does not cut the simplex with vertices $g(T_i)$. That is, we can't have the following in the stability picture:

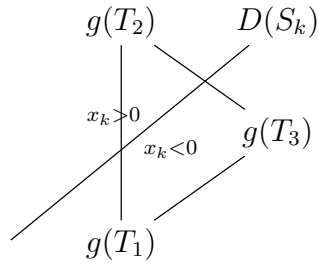
\begin{center}
\begin{tabular}{c}
\xymatrix{ & g(T_2) \ar@{-}[dd] \ar@{-}[dr] &  D(S_k) \\ & & g(T_3) \\ \ar@{-}^{x_k > 0}_{x_k < 0}[uurr]  & g(T_1) \ar@{-}[ur]  & }  
\end{tabular}
\end{center}

\subsection{$c$-vectors and Frozen Vertices}

\indent 

Recall the notion of quiver mutation from Definition \ref{def: quiver mutation}. We can provide an analogous definition of mutation in terms of the exchange matrix, which can be found in [\ref{ref: BDG matrix mutation}], as follows. 

\begin{defn}
Let $B$ be the exchange matrix for a quiver $Q$. Then for any $1 \leq k \leq n$,
the \textbf{mutation of $B$ in the direction $k$} is the matrix $\mu_k(B) = [b_{ij}]$ given by

\begin{center}

\[ b_{ij} = \begin{cases} 
          -b_{ij} & i = k \, \text{or} \, j = k \\
          b_{ij} + \text{max}(b_{ik},0)\text{max}(b_{kj},0) - \text{min}(b_{ik},0)\text{min}(b_{kj},0) & \text{otherwise}
       \end{cases}
\]

\end{center}

\end{defn}

Notice that $\mu_k(B)$ is the exchange matrix of $\mu_k(Q)$, hence is skew-symmetric. We can perform mutations of the exchange matrix in two steps as in the following example.

\begin{exmp}
Recall the exchange matrix from Example \ref{exmp: exchange matrix}. Then we can mutate at 2 using the following two step process. Since we are mutating at two, we highlight both the second row and column. 

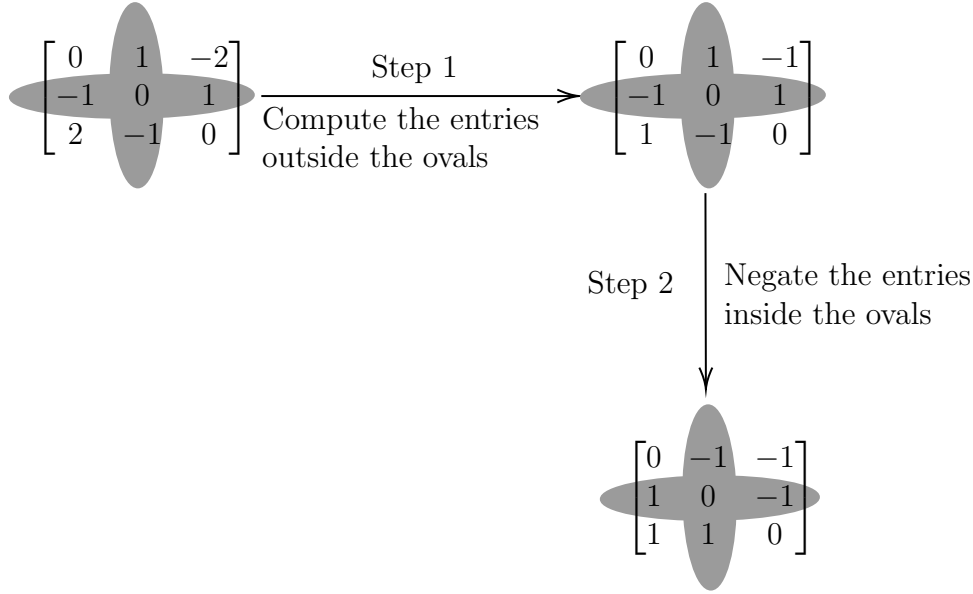
\begin{center}

\tikzset{every picture/.style={line width=0.75pt}} 

\begin{tikzpicture}[x=0.75pt,y=0.75pt,yscale=-1,xscale=1]

\draw  [color={rgb, 255:red, 155; green, 155; blue, 155 }  ,draw opacity=1 ][fill={rgb, 255:red, 155; green, 155; blue, 155 }  ,fill opacity=1 ] (166.79,103.1) .. controls (159.48,103.22) and (153.23,82.5) .. (152.82,56.81) .. controls (152.41,31.13) and (157.99,10.22) .. (165.29,10.1) .. controls (172.59,9.98) and (178.85,30.71) .. (179.26,56.39) .. controls (179.67,82.07) and (174.09,102.99) .. (166.79,103.1) -- cycle ;
\draw  [color={rgb, 255:red, 155; green, 155; blue, 155 }  ,draw opacity=1 ][fill={rgb, 255:red, 155; green, 155; blue, 155 }  ,fill opacity=1 ] (225.22,56.38) .. controls (225.28,62.46) and (197.79,67.66) .. (163.82,68) .. controls (129.85,68.34) and (102.27,63.68) .. (102.21,57.61) .. controls (102.15,51.53) and (129.63,46.33) .. (163.6,45.99) .. controls (197.57,45.65) and (225.15,50.31) .. (225.22,56.38) -- cycle ;
\draw    (229,57) -- (387.22,57) ;
\draw [shift={(389.22,57)}, rotate = 180] [color={rgb, 255:red, 0; green, 0; blue, 0 }  ][line width=0.75]    (10.93,-3.29) .. controls (6.95,-1.4) and (3.31,-0.3) .. (0,0) .. controls (3.31,0.3) and (6.95,1.4) .. (10.93,3.29)   ;
\draw    (453,106) -- (453.22,203) ;
\draw [shift={(453.22,205)}, rotate = 269.87] [color={rgb, 255:red, 0; green, 0; blue, 0 }  ][line width=0.75]    (10.93,-3.29) .. controls (6.95,-1.4) and (3.31,-0.3) .. (0,0) .. controls (3.31,0.3) and (6.95,1.4) .. (10.93,3.29)   ;
\draw  [color={rgb, 255:red, 155; green, 155; blue, 155 }  ,draw opacity=1 ][fill={rgb, 255:red, 155; green, 155; blue, 155 }  ,fill opacity=1 ] (454.79,103.1) .. controls (447.48,103.22) and (441.23,82.5) .. (440.82,56.81) .. controls (440.41,31.13) and (445.99,10.22) .. (453.29,10.1) .. controls (460.59,9.98) and (466.85,30.71) .. (467.26,56.39) .. controls (467.67,82.07) and (462.09,102.99) .. (454.79,103.1) -- cycle ;
\draw  [color={rgb, 255:red, 155; green, 155; blue, 155 }  ,draw opacity=1 ][fill={rgb, 255:red, 155; green, 155; blue, 155 }  ,fill opacity=1 ] (513.22,56.38) .. controls (513.28,62.46) and (485.79,67.66) .. (451.82,68) .. controls (417.85,68.34) and (390.27,63.68) .. (390.21,57.61) .. controls (390.15,51.53) and (417.63,46.33) .. (451.6,45.99) .. controls (485.57,45.65) and (513.15,50.31) .. (513.22,56.38) -- cycle ;
\draw  [color={rgb, 255:red, 155; green, 155; blue, 155 }  ,draw opacity=1 ][fill={rgb, 255:red, 155; green, 155; blue, 155 }  ,fill opacity=1 ] (455.79,306.1) .. controls (448.48,306.22) and (442.23,285.5) .. (441.82,259.81) .. controls (441.41,234.13) and (446.99,213.22) .. (454.29,213.1) .. controls (461.59,212.98) and (467.85,233.71) .. (468.26,259.39) .. controls (468.67,285.07) and (463.09,305.99) .. (455.79,306.1) -- cycle ;
\draw  [color={rgb, 255:red, 155; green, 155; blue, 155 }  ,draw opacity=1 ][fill={rgb, 255:red, 155; green, 155; blue, 155 }  ,fill opacity=1 ] (510.22,259.38) .. controls (510.28,265.46) and (485.71,270.63) .. (455.34,270.93) .. controls (424.97,271.24) and (400.3,266.55) .. (400.24,260.48) .. controls (400.18,254.4) and (424.75,249.23) .. (455.12,248.93) .. controls (485.49,248.62) and (510.15,253.31) .. (510.22,259.38) -- cycle ;

\draw (113,27.4) node [anchor=north west][inner sep=0.75pt]    {$\begin{bmatrix}
0 & 1 & -2\\
-1 & 0 & 1\\
2 & -1 & 0
\end{bmatrix}$};
\draw (283,35) node [anchor=north west][inner sep=0.75pt]   [align=left] {Step 1};
\draw (228,61) node [anchor=north west][inner sep=0.75pt]   [align=left] {Compute the entries\\outside the ovals};
\draw (392,144) node [anchor=north west][inner sep=0.75pt]   [align=left] {Step 2};
\draw (461,140) node [anchor=north west][inner sep=0.75pt]   [align=left] {Negate the entries\\inside the ovals};
\draw (401,27.4) node [anchor=north west][inner sep=0.75pt]    {$\begin{bmatrix}
0 & 1 & -1\\
-1 & 0 & 1\\
1 & -1 & 0
\end{bmatrix}$};
\draw (411,229.4) node [anchor=north west][inner sep=0.75pt]    {$\begin{bmatrix}
0 & -1 & -1\\
1 & 0 & -1\\
1 & 1 & 0
\end{bmatrix}$};

\end{tikzpicture}

\end{center}

\end{exmp}

We will now extend the exchange matrix of $Q$ by adding the identity matrix to the bottom of $B$. This new extended exchange matrix of $B$ will be denoted $\tilde{B} = {B\over C}$ where $C$ is the identity matrix. We can provide a purely combinatorial description of how mutation for this new extended exchange matrix works; however, it may be useful to introduce the notion of frozen vertices which can also be found in [\ref{ref: BDG matrix mutation}].

\begin{defn}
To any quiver $Q = (\{1,2,\dots, n\}, \{\alpha_1, \alpha _2, \dots \alpha_k\})$, we denote by $\check{Q}$ the quiver with $\check{Q}_0 = Q_0 \cup \{1', 2', \dots , n'\}$ and $\check{Q}_1 = Q_1 \cup \{\alpha'_i: i' \rightarrow i : i \in \{1, 2, \dots, n\}\}$. The vertices not in $Q_0$ are called \textbf{frozen vertices} since we can't mutate $\check{Q}$ at these vertices. 
\end{defn}

\begin{rem}
Sometimes the notation $\hat{Q}$ is used. The quiver $\hat{Q} = (\check{Q}_0 , Q_1 \cup \{\alpha'_i: i' \leftarrow i : i \in \{1, 2, \dots, n\}\})$.
\end{rem}

We can use this extended quiver to define mutation of the extended exchange matrix. In particular the mutated extended exchange matrix is just the exchange matrix of the mutated extended quiver.

\begin{exmp}
Let $Q = 1\rightarrow 2$. Then the extended quiver along with a mutation at vertex 1 is depicted below. Note that the frozen vertices are blue. 

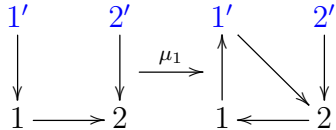
\begin{center}
\begin{tabular}{c}
\xymatrix{{\color{blue} 1'}\ar[d] & {\color{blue} 2'}\ar[d] \ar@<-20pt>^{\mu_1}[r] & {\color{blue} 1'}\ar[dr] & {\color{blue} 2'}\ar[d] \\ 1 \ar[r] & 2 & 1 \ar[u] & 2\ar[l]}
\end{tabular}
\end{center}

The extended exchange matrix along with its mutation at one is depicted below. For $i,j \leq |Q_0|$, the $ij$th entry denotes the number of arrows from $i$ to $j$ minus the number of arrows from $j$ to $i$. For $i > |Q_0|$, the $ij$th entry denotes the number of arrows from $(i-|Q_0|)'$ to $j$ minus the number of arrows from $j$ to $(i-|Q_0|)'$. For example in the mutated matrix, entry 31 is -1 since in the mutated extended quiver there are no arrows from $1'$ to 1 and there is one arrow from 1 to $1'$.

\begin{center}
\begin{tabular}{c c c}
$\begin{bmatrix} 0 & 1 \\ -1 & 0 \\ \hline 1 & 0 \\ 0 & 1 \end{bmatrix}$ & 
$\overset{\mu_1}{\rightarrow}$ & 
$\begin{bmatrix} 0 & -1 \\ 1 & 0 \\ \hline -1 & 1 \\ 0 & 1 \end{bmatrix}$
\end{tabular}
\end{center}

\end{exmp}

\begin{defn}
Let $\tilde{B} = {B\over C}$ be the extended exchange matrix associated to the cluster algebra with initial seed $(Q,x_{*})$. Then the set of \textbf{$c$-vectors} is the set consisting of the columns of $C$ along with the columns of any $C'$ that arises from a mutation of $\tilde{B}$. The matrix $C$ along with any mutation of it is called a \textbf{C-matrix}.
\end{defn}

\subsection{Maximal Green Sequences}

\indent

We are now ready to define maximal green sequences. As a remark, these sequences are not named after a person, but the color green. Keller is credited with coming up with the green/red traffic light coloring system. 

\begin{defn}
A \textbf{maximal green sequence} (MGS) is a sequence of green mutations starting with $\tilde{B} = {B \over I}$ and ending with all negative $c$-vectors. A mutation $\mu_k$ is said to be \textbf{green} if the $c$-vector $c_k$ is positive. If $c_k$ is negative, $\mu_k$ is said to be a \textbf{red} mutation.
\end{defn}

\begin{rem}
The definition of a MGS is dependent on the fact that $c$-vectors can't be both red and green at the same time; that is, they are sign-coherent. This was proven first by Derksen, Weymen, and Zelevinski for quivers with potential in [\ref{ref: DWZ sign coherence c-vectors}] then later by Gross, Hacking, Keel, and Kontsevich for general cluster algebras in [\ref{ref: GHKK sign coherence c-vectors}]. In [\ref{ref: Fu sign coherence c-vectors}], Fu defined the $c$-vectors in terms of $g$-vectors then proved the statement for all finite dimensional algebras. Fu used a fact that we will see later; that is, the sign coherence of $c$-vectors follows from the fact that they are the dimension vectors of certain indecomposable modules. In [\ref{ref: T which bricks correspond to $c$-vectors}], Treffinger classifies precisely the modules whose dimension vectors give positive $c$-vectors.
\end{rem}

\begin{conj}
For $Q$ a quiver with no oriented cycles, there are only finitely many MGS.
\end{conj}
This is known to be true for tame, affine, and wild quivers with at most three vertices and was proven by Br\"{u}stle, Dupont, and Perotin in [\ref{ref: BDG matrix mutation}], but unknown in general.

\begin{exmp}
Let $Q$ be the quiver $1\rightarrow 2$. Then we have a maximal green sequence given by 

\begin{center}
\begin{tabular}{c c c c c c c}
$\begin{bmatrix} 0 & 1 \\ -1 & 0 \\ \hline {\color{green} 1} & {\color{green} 0} \\ {\color{green} 0} & {\color{green} 1} \end{bmatrix}$ & 
$\overset{\mu_1}{\rightarrow}$ & 
$\begin{bmatrix} 0 & -1 \\ 1 & 0 \\ \hline {\color{red}-1} & {\color{green} 1 } \\ {\color{red} 0 } & {\color{green}1} \end{bmatrix}$ & 
$\overset{\mu_2}{\rightarrow}$ &
$\begin{bmatrix} 0 & 1 \\ -1 & 0 \\ \hline {\color{green} 0} & {\color{red} -1 } \\ {\color{green} 1 } & {\color{red} -1} \end{bmatrix}$ &
$\overset{\mu_1}{\rightarrow}$ &
$\begin{bmatrix} 0 & -1 \\ 1 & 0 \\ \hline {\color{red} 0} & {\color{red} -1 } \\ {\color{red} -1 } & {\color{red} 0} \end{bmatrix}$ 
\end{tabular}
\end{center}

Notice that the $c$-vectors are the dimension vectors of $S_1, P_1$, and $S_2$ or their negatives. We can also visualize this same maximal green sequence using frozen vertices as follows.

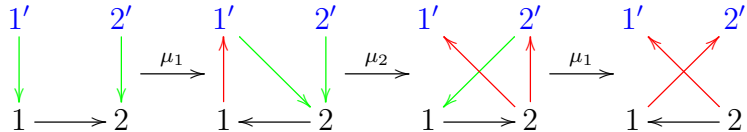
\begin{center}
\begin{tabular}{c}
\xymatrix{{\color{blue} 1'}\ar@[green][d] & {\color{blue} 2'}\ar@[green][d] \ar@<-20pt>^{\mu_1}[r] & {\color{blue} 1'}\ar@[green][dr] & {\color{blue} 2'}\ar@[green][d] \ar@<-20pt>^{\mu_2}[r]  &  {\color{blue} 1'} &  {\color{blue} 2' } \ar@[green][dl] \ar@<-20pt>^{\mu_1}[r] & {\color{blue} 1'}   & {\color{blue} 2' }  \\ 1 \ar[r] & 2 & 1 \ar@[red][u] & 2\ar[l] & 1 \ar[r] & 2 \ar@[red][u] \ar@[red][ul]  & 1 \ar@[red][ur] & 2 \ar@[red][ul] \ar[l]}
\end{tabular}
\end{center}

A third way to visualize a MGS is through so-called green paths in the stability picture, which were introduced by Igusa and Todorov in [\ref{ref: IT green paths}] and further explored by Igusa in [\ref{ref: I green paths}]. The walls indicate which $c$-vectors to mutate. For instance in the lower path, which corresponds to the MGS shown earlier in this example, we mutate the $c$-vector that gives the dimension vectors of $S_1, P_2,$ and $S_2$ in that order. This is precisely $\mu_1 \circ 
\mu_2 \circ \mu_1$ as seen above. Moreover, the below picture shows that there are only two paths from the all green region to the all red region. Therefore, there are only two maximal green sequences in this example; the lower path corresponds to the MGS shown earlier in this example and the MGS corresponding to the upper path is shown below the stability picture.

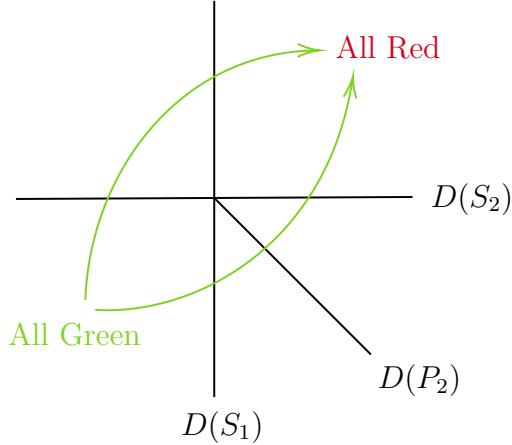
\begin{center}
\tikzset{every picture/.style={line width=0.75pt}} 

\begin{tikzpicture}[x=0.75pt,y=0.75pt,yscale=-1,xscale=1]

\draw    (330.22,60.01) -- (330.22,260.01) ;
\draw    (230.22,160.01) -- (430.22,159.01) ;
\draw    (330.22,159.51) -- (409.22,238.51) ;
\draw [color={rgb, 255:red, 126; green, 211; blue, 33 }  ,draw opacity=1 ]   (265.22,211.01) .. controls (267.21,149.32) and (310.78,85.65) .. (382.14,85.01) ;
\draw [shift={(383.22,85.01)}, rotate = 180] [color={rgb, 255:red, 126; green, 211; blue, 33 }  ,draw opacity=1 ][line width=0.75]    (10.93,-3.29) .. controls (6.95,-1.4) and (3.31,-0.3) .. (0,0) .. controls (3.31,0.3) and (6.95,1.4) .. (10.93,3.29)   ;
\draw [color={rgb, 255:red, 126; green, 211; blue, 33 }  ,draw opacity=1 ]   (270.22,216.01) .. controls (321.96,218.99) and (387.56,182.37) .. (400.04,99.26) ;
\draw [shift={(400.22,98.01)}, rotate = 98.13] [color={rgb, 255:red, 126; green, 211; blue, 33 }  ,draw opacity=1 ][line width=0.75]    (10.93,-3.29) .. controls (6.95,-1.4) and (3.31,-0.3) .. (0,0) .. controls (3.31,0.3) and (6.95,1.4) .. (10.93,3.29)   ;

\draw (312,265.4) node [anchor=north west][inner sep=0.75pt]    {$D( S_{1})$};
\draw (438,150.4) node [anchor=north west][inner sep=0.75pt]    {$D( S_{2})$};
\draw (411.22,241.91) node [anchor=north west][inner sep=0.75pt]    {$D( P_{2})$};
\draw (225,221) node [anchor=north west][inner sep=0.75pt]   [align=left] {\textcolor[rgb]{0.49,0.83,0.13}{All Green}};
\draw (390,76) node [anchor=north west][inner sep=0.75pt]   [align=left] {\textcolor[rgb]{0.82,0.01,0.11}{All Red}};
\end{tikzpicture}

\end{center}

\begin{center}
\begin{tabular}{c c c c c}
$\begin{bmatrix} 0 & 1 \\ -1 & 0 \\ \hline {\color{green} 1} & {\color{green} 0} \\ {\color{green} 0} & {\color{green} 1} \end{bmatrix}$ & 
$\overset{\mu_2}{\rightarrow}$ & 
$\begin{bmatrix} 0 & -1 \\ 1 & 0 \\ \hline {\color{green} 1} & {\color{red} 0 } \\ {\color{green} 0 } & {\color{red} -1} \end{bmatrix}$ & 
$\overset{\mu_1}{\rightarrow}$ &
$\begin{bmatrix} 0 & 1 \\ -1 & 0 \\ \hline {\color{red} -1} & {\color{red} 0 } \\ {\color{red} 0 } & {\color{red} -1} \end{bmatrix}$ 
\end{tabular}
\end{center}

\end{exmp}

\subsection{Relationship Between $c$-vectors and $g$-vectors}

\indent

As we have seen in the previous examples, the $c$-vectors are dimension vectors of certain indecomposable modules. In particular, this means they correspond to the walls in the stability picture and the $g$-vectors are vectors that lie in these walls. The following theorem formalizes these two ideas. The first part of the following theorem was done by Nakanishi and Zelevinski in [\ref{ref: NZ formula}] where the latter part was done by Fu in [\ref{ref: Fu sign coherence c-vectors}].

\begin{thm}
Let $T = \oplus T_i$ be a silting module. Then the following hold.
\begin{enumerate}
\item The $g$-vectors $g(T_i)$ and $c$-vectors $c_j$ are related by the Nakanishi-Zelevinsky formula: 
$$g(T_i) \cdot c_j = -\delta_{ij}.$$
\item $c_{j} = \pm \text{dim}M_{j}$ where $g(T_i) \in D(M_j)$ for all $i\neq j$.
\end{enumerate}
\end{thm}

\begin{proof}[Sketch of Part of Proof]
We will take $j = 1$ and assume the Nakanishi-Zelevinsky formula. By Lemma \ref{lem: Schofield}, $(T_2, \dots , T_n)$ forms an exceptional sequence. By properties of exceptional sequences, there exists a unique module $M_1$ such that $(M_1, T_2, \dots , T_n)$ is exceptional. We conclude that Hom$(T_j,M_1) = 0 = \text{Ext}(T_j,M_1)$ for all $j$. Therefore, $g(T_j) \in D(M_1) \subset (\text{dim}M_1)^{\perp}$. This implies that $g(T_j)$ is orthogonal to the dimension vector of $M_1$. By the Nakanishi-Zelevinsky formula, $c_1 = \pm \text{dim}M_1$.
\end{proof}

It is not the common notation to negate the Kronecker delta in the Nakanishi-Zelevinsky formula. We will finish this section with a brief explanation of why the negation of the Kronecker delta makes sense from a representation-theoretic perspective. \\

The initial cluster in a cluster algebra is $(x_1, x_2, \dots , x_n)$ and this corresponds to the silting pair $(0, P_1[1] \oplus P_2[1] \oplus \dots \oplus P_n[1])$ where $P_i[1]$ has $g$-vector $-e_i$. By definition of a MGS, the initial $c$-vectors are $c_i = e_i$. Therefore, the initial $g$-vectors dotted with the initial $c$-vectors satisfy the formula $g(T_i) \cdot c_j = -\delta_{ij}$. This relation is preserved under mutations since clusters are defined to be those rational functions attained from the initial cluster by a finite sequence of mutations, hence the negation of the Kronecker delta.


\eject
\section{References}

\begin{enumerate}[ {[}1{]} ]

\item Adachi, T., Iyama, O., Reiten, I. \textit{$\tau$-tilting theory} Compos. Math., 150 (3) (2014), pp. 415-452 \label{ref: AIR tau-tilting} 

\item Amiot, C. \textit{Cluster categories for algebras of global dimension 2 and quivers with potential}. Ann. Inst. Fourier, 59(6):2525–2590, 2009. \label{ref: A non hereditary cluster cat}

\item Assem, I., Simson, D., Skowronski, A., \textit{Elements of the Representation Theory of Associative Algebras: Volume 1: Techniques of Representation Theory.}. Cambridge University Press. 2006. \label{ref: blue book}

\item Bridgeland, T. \textit{Scattering diagrams, Hall algebras and stability conditions} Algebr. Geom., 4 (5) (2017), pp. 523-561 \label{ref: B scattering diagrams}

\item Broomhead, N., Pauksztello, D., Ploog, D. \textit{Discrete derived categories II: the silting pairs CW complex and
the stability manifold}, J. Lond. Math. Soc. (2) 93 (2016), 273-300. \label{ref: BPP silting pairs}

\item Br\"{u}stle, T., Dupont, G., Perotin, M. \textit{On Maximal Green Sequences}, May 2012, International Mathematics Research Notices, 2014 (16), DOI:10.1093/imrn/rnt075 \label{ref: BDG matrix mutation}

\item Br\"{u}stle, T., Smith, D., Treffinger, H. \textit{Wall and chamber structure for finite-dimensional algebras} Advances in Mathematics Volume 354, 1 October 2019, 106746 \label{ref: BT wall and chamber structure} 

\item Buan, A. B, Marsh, R., Reineke, M., Reiten, I., Todorov, G. \textit{Tilting theory and cluster combinatorics} Adv. Math., 204 (2) (2006), pp. 572-618 \label{ref: BMRRT categorification of cluster algebras} 

\item  Caldero P., Chapoton, F. \textit{Cluster algebras as Hall algebras of quiver representations}, Comment. Math. Helv. 81 (2006), no. 3, 595-616. \label{ref: CC equation}

\item Cerulli Irelli, Giovanni. \textit{Three Lectures on Quiver Grassmannians}. Preprint, arXiv:2003.08265 [math.AG], 2020. \label{ref: CI quiver grassmannians}

\item Chapoton, F., Fomin, S., Zelevinsky, A. \textit{Polytopal realizations of generalized associahedra}. Canad. Math. Bull. 45, no. 4, 537–566, 2002. \label{ref: CFZ original compatibility graph}

\item Crawley-Boevey, W. \textit{Lectures on Representations of Quivers}, Bielefeld University, \url{https://www.math.uni-bielefeld.de/~wcrawley/quivlecs.pdf} \label{ref: Bill's notes}

\item Dehy, R., Keller, B. \textit{On the combinatorics of rigid objects in 2-Calabi-Yau categories} Int. Math. Res. Not. IMRN (11) (2008), Article rnn029 \label{ref: DK rep theory g-vectors}

\item Demonet, L., Iyama, O., Jasso, G. \textit{$\tau$-tilting finite algebras, bricks, and $g$-vectors} Int. Math. Res. Not., 2019 (3) (07 2017), pp. 852-892 \label{ref: DIJ triangulation of sphere}

\item Derksen, H., Weyman, J., Zelevinsky, A. \textit{Quivers with potentials and their representations II: applications to cluster algebras} J. Am. Math. Soc., 23 (3) (2010), pp. 749-790 \label{ref: DWZ sign coherence c-vectors}


\item Fomin, S., Zelevinsky, A. \textit{Cluster algebras I: Foundations}, J. Amer. Math. Soc. 15 (2002), 497–529. \label{ref: FZ cluster algbras}

\item Fomin, S., Zelevinsky, A. \textit{Cluster algebras. II: Finite type classification} Invent. Math., 154(1):63–121, 2003. \label{ref: FZ finite type classification}

\item Fomin, S., Zelevinsky, A. \textit{Cluster algebras. IV. Coefficients}. Compos. Math., 143(1):112–164, 2007. \label{ref: FZ g-vectors}

\item Fu, C. \textit{$c$-Vectors via $\tau$-tilting theory} J. Algebra, 473 (2017), pp. 194-220 \label{ref: Fu sign coherence c-vectors}

\item Gabriel, P. \textit{Unzerlegbare Darstellungen. I} Manuscripta Math., 6:71–103; correction, ibid. 6 (1972), 309, 1972. \label{ref: Gabriel's theorem 1}

\item Gabriel, P. \textit{Indecomposable representations. II} Sympos. math. 11, Algebra commut., Geometria, Convegni 1971/1972, 81-104 (1973). \label{ref: Gabriel's theorem 2}

\item Gross, M., Hacking, P., Keel, S., Kontsevich, M. \textit{Canonical bases for cluster algebras} J. Am. Math. Soc., 31 (2) (11 2017), pp. 497-608 \label{ref: GHKK sign coherence c-vectors}

\item Happel, D., Ringel, C. M. \textit{Almost complete tilting modules}. Trans. Amer. Math. Soc., Vol. 274, No. 2 (Dec., 1982), pp.
399-443. \label{ref: HR happel-ringel lemma}

\item Happel, D., Unger, L. \textit{Almost complete tilting modules}. Proc. Amer. Math. Soc., 107(3):603–610, 1989. \label{ref: HU two completions}

\item Hoshino, M., Kato, Y., Miyachi, J. \textit{On $t$-structures and torsion theories induced by compact objects}, J. Pure Appl. Algebra 167 (2002) 15–35 \label{ref: HKM Two term silting complexes}

\item H\"{u}gel, L. A. \textit{Silting Objects}, Bull. of the Lon. Math. Soc. Volume 51, Issue 4 August 2019 Pages 658-690 \label{ref: H silting objects}

\item Igusa, K. \textit{Linearity of Stability Conditions} Communications in Algebra, 48(4), June 2017 DOI:10.1080/00927872.2019.1705466 \label{ref: I green paths}


\item Igusa, K., Orr, K., Todorov, G., Weyman, J. \textit{Cluster complexes via semi-invariants} Compos. Math., 145 (4) (2009), pp. 1001-1034 \label{ref: IOTW c-vectors are walls}

\item Igusa, K., Orr, K., Todorov, G., Weyman, J. \textit{Modulated semi-invariants}  Preprint, arXiv:1507.03051 [math.RT], 2015 \label{ref: IOTW rigidity is open} 

\item Igusa, K., Schiffler, R., \textit{Exceptional Sequences and Clusters}. Jounal of Algebra 323 (2010) 2183-2202 \label{ref: IS Clusters form exceptional collections} 

\item Igusa, K., Todorov, G. \textit{Picture groups and maximal green sequences} Electronic Research Archive 2021, Volume 29, Issue 5: 3031-3068. doi: 10.3934/era.2021025 \label{ref: IT green paths}

\item  Ingalls, C., Thomas, H. \textit{Noncrossing partitions and representations of quivers} Compos. Math., 145(6):1533–1562, 2009 \label{ref: IT support tilting}

\item  Keller, B., Vossieck, D. \textit{Aisles in derived categories}, Bull. Soc. Math. Belg. Ser. A 40 (1988), no. 2, 239–25 \label{ref: KV intro silting}

\item King, A. D. \textit{Moduli of representations of finite dimensional algebras} QJ Math., 45 (4) (1994), pp. 515-530 \label{ref: K stability conditions}

\item Lee K., Schiffler, R. \textit{Positivity for cluster algebras}. Annals of Mathematics, SECOND SERIES, Vol. 182, No. 1 (July, 2015), pp. 73-125 (53 pages)\label{ref: LS positivity}

\item Marsh, R., Reineke, M., Zelevinsky, A. \textit{Generalized associahedra via quiver representations} Trans. Amer. Math. Soc. 355 no.10, 4171-4186, 2003 \label{ref: MRZ compatibility graph}

\item Nakanishi, T., Zelevinsky, A. \textit{On tropical dualities in cluster algebras}, Contemp. Math. 565(2012), 217-226. \label{ref: NZ formula}

\item Schiffler, R. \textit{Quiver Representations}, CMS Books in Mathematics, Springer International Publishing (2014) \label{ref: Schiffler Quiver Reps} 

\item Schofield, A. \textit{General Representations of Quivers} Proc. of The Lon. Math. Soc, Vol. 65, (1992), pp. 46-64 \label{ref: S general reps}

\item Skowroński, A. \textit{Regular Auslander-Reiten components containing directing modules}. Proc. Amer. Math. Soc., 120(1):19–26, 1994. \label{ref: S tilting modules}

\item Thomas, H. \textit{An Introduction to the Lattice of Torsion Classes} Bull. of the Iran. Math. Soc., 47, pp 35-55 (2021) \label{ref: Th survey on lattice of torsion classes}

\item Treffinger, H., \textit{On sign-coherence of $c$-vectors} Journal of Pure and Applied Algebra, Volume 223, Issue 6, June 2019, Pages 2382-2400 \label{ref: T which bricks correspond to $c$-vectors}

\item Treffinger, H. \textit{$\tau$-tilting theory -- An introduction} Preprint, arXiv:2106.00426 [math.RT], 2022 \label{ref: T survey} 

\item Vejdemo-Johansson, M. \textit{Sketches of a platypus: persistent homology and its algebraic foundations} Preprint, arXiv:212.5398 [math.AT], 2013 \label{ref: V-J survey on persistent homology}


\end{enumerate}

\end{document}